\documentclass[a4paper,11pt]{article}
\usepackage[ansinew]{inputenc}
\usepackage[spanish,english]{babel}
\usepackage{amsmath,amsfonts,amssymb,amsthm}
\usepackage{mathtools}
\usepackage{graphicx,subfigure}
\usepackage[usenames,dvipsnames]{color}
\usepackage{lineno}
\usepackage{comment}
\usepackage{enumerate}
\usepackage[colorlinks=true,citecolor=black,linkcolor=black,urlcolor=blue]{hyperref}
\usepackage{amscd,latexsym}
\usepackage{authblk}
\usepackage{cite}
\usepackage{soul}
\usepackage{xcolor}
\usepackage[letterpaper,top=2cm,bottom=2cm,left=3cm,right=3cm,marginparwidth=1.75cm]{geometry}
\usepackage[T1]{fontenc}  
\usepackage[inline]{enumitem}
\usepackage{float}

\theoremstyle{plain}
\newtheorem{theorem}{Theorem}
\newtheorem{lemma}{Lemma}

\newtheorem{problem}{Open problem}

\newcommand{\cn}{\left\lceil \frac{2n}3\right\rceil}
\newcommand{\cnp}{\left\lceil \frac{2(n-3)}3\right\rceil}

\date{}

\title{On the fault-tolerant metric dimension of maximal outerplanar graphs}

\author[1]{Carmen Hernando}
\author[1]{Montserrat Maureso}
\author[1]{Merc\`e Mora}
\author[2]{Javier Tejel}
\affil[1]{Departament de Matem\`atiques, Universitat Polit\`ecnica de Catalunya, Barcelona, Spain.  \{carmen.hernando, montserrat.maureso, merce.mora\}@upc.edu}
\affil[2]{Departamento de M\'etodos Estad\'\i sticos and IUMA, Universidad de Zaragoza, Zaragoza, Spain. jtejel@unizar.es}

\begin{document}

\maketitle

\begin{abstract}
Given a graph $G$, a set $W$ of vertices of $G$ is a resolving set if, for any pair of vertices $z_1$ and $z_2$ of $G$, there is a vertex $w\in W$ such that $d(w,z_1)\ne d(w,z_2)$, where $d(u,v)$ denotes the length of a shortest path between the vertices $u$ and $v$. 
The metric dimension of a graph $G$, $\beta(G)$, is the minimum cardinality of a resolving set of $G$. A resolving set $W$ is fault-tolerant if $W\setminus \{w\}$ is a resolving set for any $w\in W$. The fault-tolerant metric dimension of a graph $G$, denoted by $\beta'(G)$, is the minimum cardinality of a fault-tolerant resolving set of $G$.

In this paper, we study the fault-tolerant metric dimension problem in maximal outerplanar graphs. In particular, we show that there are exactly two maximal outerplanar graphs with $\beta'(G)=3$. 
We also show that $4\le \beta'(G) \le 6$ for any maximal outerplanar graph $G$ of order at least 7 and metric dimension 2. Moreover, we establish the upper bound $\beta'(G) \le \left\lceil \frac{2n}{3}  \right\rceil$ for maximal outerplanar graphs $G$ of order $n\ge 5$, by building in linear time a fault-tolerant resolving set of cardinality $\cn$.


\end{abstract}

\section{Introduction}\label{se:intro}

The metric dimension of a graph, a concept introduced independently by Slater~\cite{Slater} and by Harary and Melter~\cite{Har76}, is a classical combinatorial parameter that measures how well the vertices of a graph can be uniquely identified by distances to a set of landmark vertices. 
These landmarks are typically used to detect failures or intruders at the nodes of a network.

Formally, the notion of metric dimension is defined as follows. Given a graph $G = (V, E)$, let $d(u,v)$ denote the length of a shortest path in $G$ from vertex $u$ to vertex $v$. 
A vertex $w$ \emph{resolves} a pair of vertices $u$ and $v$, if $d(u, w) \ne d(v, w)$. A subset $W \subseteq V$ is called a \emph{resolving set} of $G$ if for every pair of distinct vertices $u, v \in V$, there exists a vertex $w \in  W$ such that $w$ resolves $u$ and $v$. 
Thus, if $r(y|W) = (d(y,w_1),\ldots, d(y,w_k))$ is the vector of distances from a vertex $y$ to the vertices of $W=\{w_1, \ldots , w_k\}\subseteq V(G)$, then $W$ is a resolving set if and only if $r(u|W)\not= r(v|W)$ whenever $u\not= v$. 
The {\em metric dimension} of $G$, denoted by $\beta(G)$, 
is the minimum cardinality among all resolving sets of $G$. A resolving set with cardinality $\beta(G)$ is called a \emph{metric basis} of $G$.

In general, computing the metric dimension or finding a metric basis of a graph is a very challenging problem. We refer the reader to~\cite{Gao25, Till23} for recent surveys on the metric dimension of graphs and its numerous applications in various fields of science and technology.

In some applications, such as network navigation or sensor placement, robustness is crucial: if one or more landmarks fail, the identification property should still hold, so that any intruder or failure can be detected by the remaining landmarks. This motivates the notion of fault-tolerant metric dimension, which was introduced by Hernando et al.~\cite{Her08}. 
A subset $W\subseteq V$ of $G$ is a \emph{fault-tolerant resolving set}, FTRS for short, if, for every $w \in W$, $W\setminus \{w\}$ is a resolving set. In other words, $W$ is a FTRS if every pair of vertices of $G$ is resolved by at least two vertices of $W$. A FTRS $W$ of $G$ with minimum cardinality is called a \emph{fault-tolerant metric basis} (FTMB for short) of $G$, and the \emph{fault-tolerant metric dimension} (FTMD for short) of $G$, denoted by $\beta'(G)$, is defined as this minimum cardinality.

In the last two decades, intensive research on the FTMD problem has been conducted, primarily aimed at determining exact or approximate values of this parameter for different families of graphs. In particular, the FTMD problem has been studied for trees~\cite{Her08}, circulant graphs~\cite{Huo20,Bas20,Saha22}, generalized Petersen graphs and sunlet graphs~\cite{Bash22}, the cube of the path graph~\cite{Saha21}, generalized prism graphs~\cite{Huo20, Fah25}, antiprism graphs~\cite{Shar25}, the Cartesian product of graphs~\cite{Gee17}, two and three dimensional grids~\cite{BY19,MST22}, certain plane rotationally-symmetric graphs~\cite{Bash21}, arithmetic graphs~\cite{Sar24}, fan and corona graphs~\cite{Cab21, Estr16bis}, the double fan graph~\cite{Rah22}, generalized wheels~\cite{Zhe21}, several ladder graphs~\cite{Wang21}, biswapped networks~\cite{Ass25}, butterfly, Benes and silicate networks~\cite{Hay20, Pra22}, some families of convex polytopes~\cite{Sid19,Zhe21, Man25}, 
and several families of chemical structures~\cite{Som19, Nad22, Aze21}. For additional results on the FTMD problem and its variants, see~\cite{Gen25, Knor25, Jav09, Khan25, Liu21}.

FTMD is also known in the literature as 2-metric dimension~\cite{Estr15}. A subset $W\subseteq V$ of $G$ is a {\em $k$-resolving set} if, for every pair of distinct vertices $u,v\in V$, there exist at least $k$ vertices in $W$
that  resolve the pair $u$ and $v$. The {\em $k$-metric dimension} of $G$ is the minimum cardinality of a $k$-resolving set. For several results concerning the $k$-metric dimension, the reader is referred to~\cite{Estr15, Estr16, Estr16bis, Yero17,Corr21, Sch23} and the references therein.

In this paper, we focus on studying the FTMD of maximal outerplanar graphs. A graph $G$ is {\em outerplanar} if it can be drawn in the plane without edge crossings and with all vertices belonging to the unbounded face. A {\em maximal outerplanar graph}, MOP for short, is an outerplanar graph such that the addition of any edge produces a non-outerplanar graph. The main results obtained in this paper are the following.
\begin{itemize}
  \itemsep 0 pt
  \item There are exactly two MOPs $G$ with $\beta'(G)=3$.
  \item If $G$ is a MOP of order at least 7 with $\beta(G)=2$, then $4\le \beta'(G) \le 6$.
  \item If $G$ is a MOP of order $n\ge 5$, then $\beta'(G) \le \left \lceil \frac{2n}{3} \right \rceil$. Furthermore, a FTRS of size $\cn$ can be built in linear time.
\end{itemize}

The paper is organized as follows. Section~\ref{sec:three} is devoted to studying MOPs $G$ with $\beta'(G)=3$, and Section~\ref{sec:four} to establish lower and upper bounds of $\beta'(G)$ for a MOP $G$ with $\beta(G)=2$. 
In Section~\ref{sec:upper}, we prove that $\beta'(G) \le \left \lceil \frac{2n}{3} \right \rceil$, if $G$ is a MOP of order $n\ge 5$. Finally, in Section~\ref{sec:conclusion}, we present some conclusions and several open problems. The rest of this section reviews several well-known properties of MOPs.

\subsection{Basic properties of MOPs}\label{sec:basicmops}

\begin{figure}[!htb]
	\centering
	\subfigure[A MOP of order 17 that contains two separating triangles and four vertices of degree 2. One of these four vertices is adjacent to a separating triangle. The dual graph is dotted.]{
		\includegraphics[scale=1, page=1]{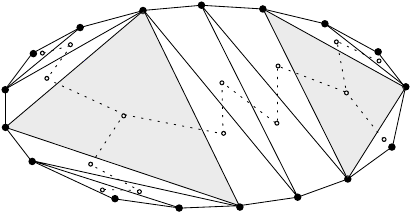}
		\label{fig:mopexample}
	}~~~~~~~~~~~
	\subfigure[A zigzag of order 12. Any of the pairs $\{u, v\}$ and $\{u', v'\}$ forms a metric basis.]{
		\includegraphics[scale=0.9, page=2]{Figures.pdf}
		\label{fig:zigzag}
	}
    \caption{Examples of MOPs.} \label{fig:mops}
\end{figure}

A MOP $G$ is biconnected, Hamiltonian, and always admits a plane embedding such that all vertices belong to the unbounded face and every bounded face is a triangle. 
Thus, $G$ can be seen as a triangulation of a simple (or convex) polygon; see Figure~\ref{fig:mopexample} for an illustration. The boundary of the unbounded face of $G$ is called {\em the boundary} of $G$.  
Every edge on the boundary of $G$ (called {\em boundary edge}) belongs to exactly one triangle of $G$, while any other edge (called {\em diagonal}) belongs to two triangles. A triangle consisting of three diagonals is a {\em separating triangle}, and the degree of its vertices in $G$ is at least four. 

The dual graph $T=(V', E')$ of a MOP $G$ (vertices of $V'$ correspond to faces in $G$ and two vertices of $T$ are connected if and only if their corresponding faces in $G$ are adjacent) is a tree, whose vertices have degree at most 3; see Figure~\ref{fig:mopexample}.  From this duality, the next properties of $G$ readily follow. 

\begin{enumerate}[label=\textbf{(B\arabic*)}, leftmargin=*]
		
	\item\label{Basic1} $G$ has at least two vertices of degree 2. Moreover, if $n\ge 4$, then the graph obtained by removing a vertex of degree 2 is also a MOP. 

	\item\label{Basic2}   The number of vertices of degree 2 in $G$ is equal to the number of separating triangles of $G$ plus 2.

    \item\label{Basic3}  Every vertex of degree 2 of $G$ has either a neighbor of degree 3 or the diagonal joining its two neighbors belongs to a separating triangle.

    \item\label{Basic4}   If $n\ge 5$, a vertex of degree 3 of $G$ is adjacent to at most one vertex of degree 2. 

    \item\label{Basic5}   If $n\ge 4$ and $G$ contains no separating triangles, then $G$ has at least two vertices of degree 2 and at least two vertices of degree 3.

\end{enumerate}

For example, the MOP shown in Figure~\ref{fig:mopexample} has two separating triangles and four vertices of degree 2. Three of these vertices have a neighbor of degree 3, and these neighbors are all different. There is only one vertex of degree 2 whose neighbors belong to a separating triangle.

A special MOP that we will use later is the zigzag; see Figure~\ref{fig:zigzag}. A \emph{zigzag} is either a triangle or a MOP of order $n\ge 4$ with two vertices of degree 2, two vertices of degree 3, and the remaining vertices, if any, of degree 4. Notice that by Property~\ref{Basic2}, a zigzag cannot contain any separating triangles. Moreover, if the zigzag has at least 4 vertices, there are two different edges with one endvertex of degree $2$ and the other of degree $3$. It is easy to check that the endvertices of such edges form a metric basis.

\section{MOPs with small FTMD}\label{sec:basic}

From the definition of FTMD, the inequality $\beta'(G)\geq \beta(G)+1$ trivially holds for any graph $G$. Moreover, 
$\beta(G)\ge 2$ for any MOP $G$, since $G$ is not a path. Therefore,  $\beta'(G)\ge 3$ for any MOP $G$. 
In this section, we show that there are only two MOPs whose FTMD is 3 and study the possible values of the FTMD for MOPs with metric dimension 2. 

To this end, we first recall the characterization of MOPs with metric dimension 2 given in~\cite{Clav21}, based on an embedding into the strong product of two paths. 
We refer the reader to that paper for further details.

The strong product of two paths of order $n$, $P_n\boxtimes P_n$, has the Cartesian product $[0,n-1]\times [0,n-1]$ as set of vertices and two different vertices $(i,j)$ and $(i',j')$ are adjacent if and only if $|i'-i|\le 1$ and $|j'-j|\le 1$. This graph can be represented in the plane by identifying the vertex $(i,j)$ with the point with Cartesian coordinates $(i,j)$.

Let $G$ be a MOP with metric dimension 2 and let $B=\{u,v\}$ be a metric basis of $G$. Then, $G$ admits an embedding $G^*$ in $P_n\boxtimes P_n$ as follows  (see an example in Figure~\ref{fig:embedded_2}).
Every vertex $w$ of $G$ is identified with the vertex of $P_n\boxtimes P_n$ having Cartesian coordinates $r(w|B)=(d(w,u),d(w,v))$. 
Hence, if $d(u,v)=d$, then $u$ is identified with $(0,d)$, $v$ is identified with $(d,0)$ and the vertices lying on the shortest path from $u$ to $v$ are identified with $(i, d-i)$, $i\in \{0,\dots ,d\}$. 
For $i=0, \ldots, d-1$, the triangle with vertices  $(i,d-i)$, $(i+1, d-i-1)$ and $(i+1,d-i)$ belongs to $G^*$. We refer to the subgraph induced by these triangles as the \emph{spine} of ${G^*}$, 
as shown in gray in 
Figure~\ref{fig:embedded_2}.

\begin{figure}[!htb]
	\centering
    \subfigure[On the left, the graphs in $\mathcal{H}$. Solid edges are mandatory. On the right, how to glue them to the spine. Also, a zigzag can be glued to the edges $(0,d)(1,d)$ and $(d,0)(d,1)$.]{
        \includegraphics[scale=0.48, page=3]{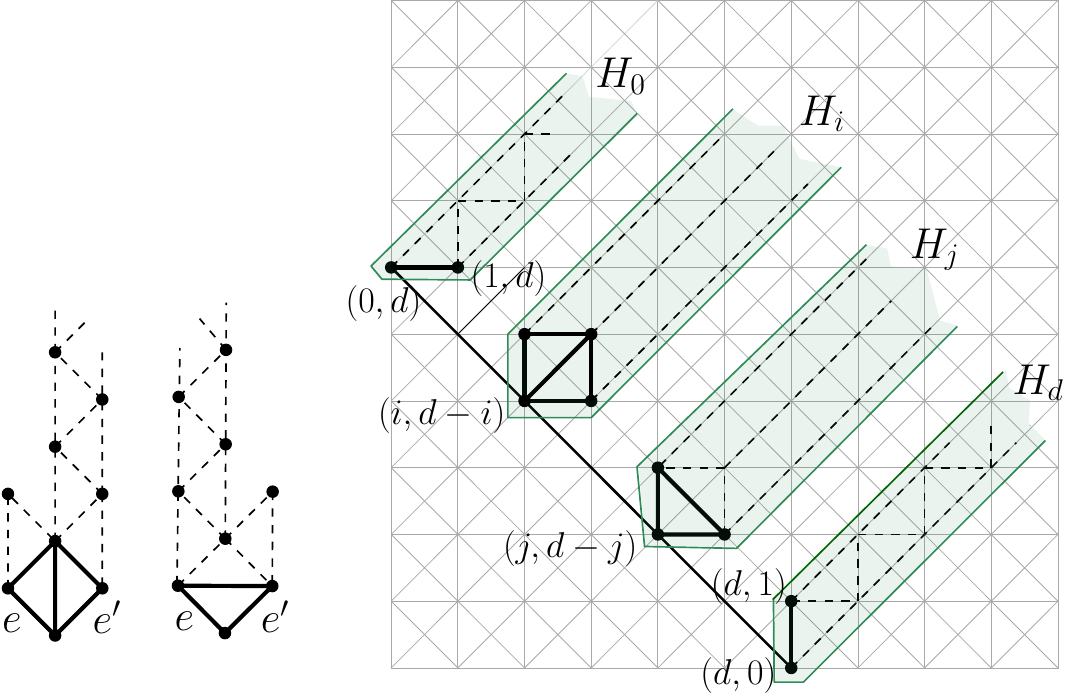}
		\label{fig:embedded_1}
	}~~
	\subfigure[An embedding $G^*$ of a MOP $G$ with metric dimension $2$. In gray, the spine. A zigzag and some graphs from $\mathcal{H}$ are glued to the spine.
    ]{
        \includegraphics[scale=0.48, page=4]{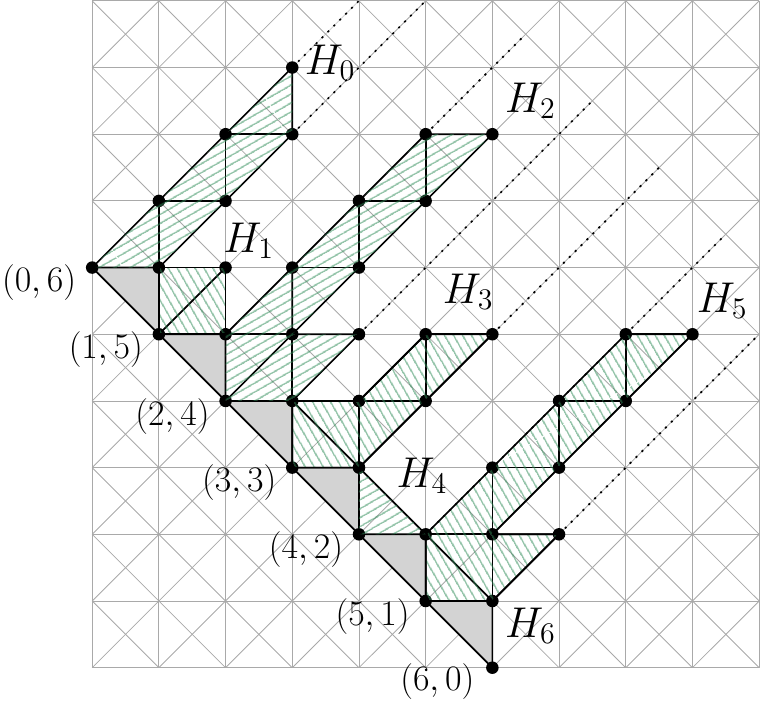}
		\label{fig:embedded_2}
   }
    \caption{Embedding a MOP $G$ with metric dimension 2 into $P_n\boxtimes P_n$.
    }
    \label{fig:embedded}
\end{figure}

Now, let $\mathcal{H}$ be the family of MOPs   
with two distinguished edges, $e$ and $e'$, 
as shown in Figure~\ref{fig:embedded_1}. 
These graphs consist of either (i) only a triangle or (ii) a MOP of order 4 and optionally a zigzag and/or a triangle attached to consecutive edges of the MOP.
Let $i\in \{1,\dots, d-1\}$ and let $e_i$ and $e_i'$ be the two edges of the spine of $G^*$ incident to $(i,d-i)$, not belonging to the shortest path between $u$ and $v$.
For each $i\in \{1,\dots, d-1\}$, a subgraph belonging to the family $\mathcal{H}$ must be glued to the spine 
as a subgraph of $G^*$, 
by identifying the edge
$e_i$ with $e$
and the edge $e'_i$ with $e'$, or vice versa.
We refer to this subgraph of $G^*$ as $H_i$. 
Also, a zigzag can be glued to the edges $(0,d)(1,d)$ and $(d,0)(d,1)$, as shown in Figure~\ref{fig:embedded_1}.
We refer to these zigzags as $H_0$ and $H_d$, respectively. 
If no zigzag is glued to the edge $(0,d)(1,d)$ (resp. $(d,0)(d,1)$), we consider that $H_0$ (resp. $H_d$) is the subgraph of order 2 induced by this edge. 
In addition, the vertices of $G^*$ not in the spine must belong to exactly one of these subgraphs  $H_0,\dots, H_d$ (so, two distinct subgraphs $H_i$, $H_j$ 
do not share any vertex, except the vertex $(i+1,d-i)$ when $j=i+1$).
See an example in Figure~\ref{fig:embedded_2}.

From this embedding $G^*$ of $G$, the next properties easily follow.

\begin{enumerate}[label=\textbf{(P\arabic*)}, leftmargin=*]
		
	\item\label{pro:p1}  The vertices $u$ and $v$ have degree 2 or 3. 

	\item\label{pro:p2}  There is only a shortest path between $u$ and $v$. Moreover, the vertices of this shortest path are consecutive on the boundary of $G$ and its internal vertices have degree 4 or 5. 

    \item\label{pro:p3}  A pair of vertices $z_1$, $z_2$ is resolved by $u$ if and only if 
  $z_1$ and $z_2$ have distinct $ x$-coordinates in $G^*$, i.e., they lie on different vertical lines. 

    \item\label{pro:p4}  A pair of vertices $z_1$, $z_2$ is resolved by $v$ if and only if 
  $z_1$ and $z_2$ have distinct $ y$-coordinates in $G^*$, i.e., they lie on different horizontal lines. 

\end{enumerate}

We remark that when computing the FTMD of a MOP with metric dimension 2, the vertices of degree 5 on the shortest path between the two vertices of a metric basis will play an important role.  Given a metric basis $B$, we denote the number of vertices of degree 5 by $r(G)$.  
Notice that $r(G)$ can vary, depending on the chosen metric basis $B$.

\subsection{MOPs with FTMD equal to three}\label{sec:three}

We begin by analyzing all MOPs of order at most 6 (see Figure~\ref{fig:littleMOPs}).
It is easy to check that all of them have metric dimension 2
and FTMD equal to 4, except for the only MOP of order 3 (Figure~\ref{fig:littleMOPs_1})  
and one MOP of order 6 
(Figure~\ref{fig:littleMOPs_6}), which have FTMD equal to 3. 

\begin{figure}[!htb]
	\centering
	\subfigure[$\beta'(G)=3$.]{
		\includegraphics[scale=0.78, page=5]{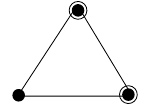}
		\label{fig:littleMOPs_1}
	}~
    \subfigure[$\beta'(G)=4$.]{
	   \includegraphics[scale=0.78, page=6]{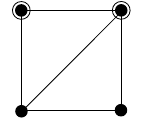}
	   \label{fig:littleMOPs_2}
	}~
	\subfigure[$\beta'(G)=4$.]{
		\includegraphics[scale=0.78, page=7]{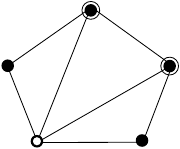}
		\label{fig:littleMOPs_3}
	}~
	\subfigure[$\beta'(G)=4$.]{
		\includegraphics[scale=0.78, page=8]{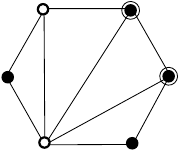}
		\label{fig:littleMOPs_4}
	}~
	\subfigure[$\beta'(G)=4$.]{
		\includegraphics[scale=0.78, page=9]{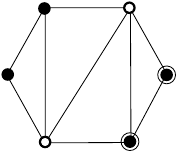}
		\label{fig:littleMOPs_5}
	}~
	\subfigure[$\beta'(G)=3$.]{
		\includegraphics[scale=0.78, page=10]{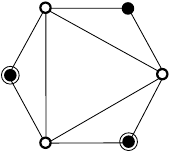}
		\label{fig:littleMOPs_6}
	}
        \caption{For all MOPs $G$ of order at most 6, $\beta(G)=2$. 
        The circled vertices form a metric basis. These vertices, together with the other black vertices, form a FTMB of the MOP.}   \label{fig:littleMOPs}
\end{figure}

\begin{lemma}\label{lem:T=3}
Let $G$ be a MOP and let $S$ be the set of vertices of $G$ of degree 2 or 3. If $\beta'(G)=3$, then $|S|=3$.
\end{lemma}

\begin{proof}
Let $W=\{u,v,w\}$ be a FTMB of $G$. 
Since any two vertices in $W$ form a metric basis of $G$, the vertices $u, v$ and $w$ have degree 2 or 3 by Property~\ref{pro:p1}. Thus, $|S|\ge 3$.
Now, suppose that $|S|\ge 4$. In such a case, there is a vertex $x$ of degree 2 or 3 that does not belong to $W$. 
Suppose that $u,v,w$ appear in this order along the boundary of $G$
and assume, without loss of generality, that $x$ is between $u$ and $v$. 
Thus, along the boundary of $G$, $x$ is between $u$ and $v$, and $w$ is between $v$ and $u$. 
Since $w$ and $x$ have degree 2 or 3, $u$ and $v$ cannot form a metric basis by Property~\ref{pro:p2}, contradicting that  $W$ is a  FTMB.
Therefore, $|S|=3$.
\end{proof}

\begin{theorem} \label{the:betap3}
Let $G$ be a MOP of order at least 3. Then, $\beta'(G)=3$ if and only if $G$ is one of the MOPs shown in Figures~\ref{fig:littleMOPs_1} and~\ref{fig:littleMOPs_6}. 
\end{theorem}

\begin{proof}
The only MOPs of order at most 6 with FTMD equal to 3 are precisely the MOPs shown in Figures~\ref{fig:littleMOPs_1} and~\ref{fig:littleMOPs_6}. 
Thus, it only remains to prove that a MOP of order at least 7 has FTMD at least 4.

Let $G$ be a MOP of order at least 7 and assume, to the contrary, that $\beta'(G)=3$.
By Lemma~\ref{lem:T=3}, $G$ has exactly three vertices of degree 2 or 3. 
Therefore, $G$ has at most one separating triangle by Property~\ref{Basic2}. Furthermore, by Property~\ref{Basic5}, $G$ must contain separating triangles because otherwise, $G$ would have at least four vertices of degree 2 or 3. As a consequence, $G$ contains exactly one separating triangle. In such a case, $G$ has exactly three vertices of degree 2 and no vertices of degree 3. By Property~\ref{Basic2}, the diagonals joining the two neighbors of each vertex of degree 2 belong to a separating triangle, implying that $G$ has order 6, which contradicts our assumption. 
\end{proof}

\subsection{Tight bounds on the FTMD of MOPs with metric dimension 2}
\label{sec:four}

In this section, we prove tight lower and upper bounds on the FTMD of MOPs with metric dimension 2. Specifically, we prove that $4\le \beta'(G)\le 6$ if $G$ is a MOP with metric dimension 2, except for the two MOPs illustrated in Figures~\ref{fig:littleMOPs_1} and~\ref{fig:littleMOPs_6}. 
Moreover, we sharpen these bounds in terms of $r(G)$, after choosing a metric basis $B$ of $G$. 
Throughout this section, we assume that a MOP $G$ has metric dimension 2, $B=\{u,v\}$ is a metric basis of $G$ with $d(u,v)=d$, and $G$ is embedded on $P_n\boxtimes P_n$ 
as explained previously. Moreover, given a vertex $z\in G$, we use $\deg(z)$ to denote its degree in $G$. 

We begin with the following special case. 

\begin{lemma}\label{lem:AuAvx0}
Let $G$ be a MOP of order $n\ge 7$, with $\beta(G)=2$, and let $\{u,v\}$ be a metric basis of $G$. If $d(u,v)=1$, then $\beta'(G)=4$.
\end{lemma}

\begin{proof}
Since $d(u,v)=1$, the spine of $G$ is only a triangle. 
Then $G$ is a zigzag or the graph shown in Figure~\ref{fig:lema2} (left) or its symmetric.
If $G$ is a zigzag, the set formed by the four vertices of degree 2 and 3 is obviously a FTRS. Thus, $\beta'(G)\le 4$ and,
by Theorem~\ref{the:betap3}, $\beta'(G)= 4$.

In the case of the MOP shown in Figure~\ref{fig:lema2} (left), 
it can be easily checked that $\{u',v'\}$ is also a metric basis, by properly embedding the graph into $P_n\boxtimes P_n$ (see Figure~\ref{fig:lema2} (right)).
Hence, the set $\{u,v,u',v'\}$ is a FTRS of $G$. Thus, $\beta'(G)\le 4$ and,
by Theorem~\ref{the:betap3}, $\beta'(G)= 4$.
\end{proof}

\begin{figure}[!htb]
	\centering	
  \includegraphics[scale=0.45, page=11]{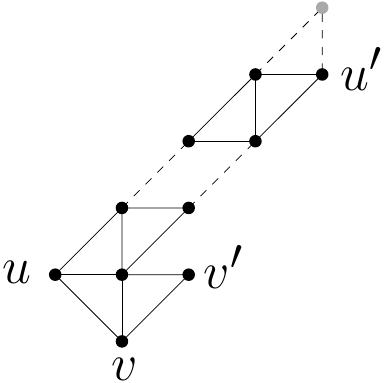}
 \hspace{20mm}
  \includegraphics[scale=0.45, page=12]{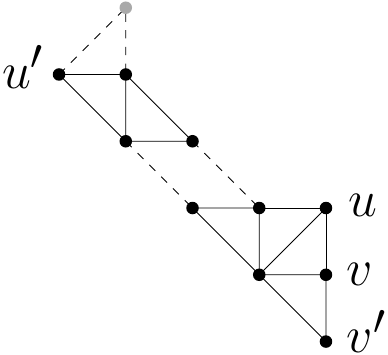}
    \caption{A MOP $G$ with $d(u,v)=1$, distinct from a zigzag, and two different embeddings of $G$ into $P_n\boxtimes P_n$. }
    \label{fig:lema2}
\end{figure}

Before proving new results, we introduce some notation for $d\ge 2$ (see Figure~\ref{fig:definition}).
The vertex $(1,d-1)$ is denoted by $u_1$, the vertex $(1,d)$ by $u_2$, and the vertex $(1,d+1)$ of $P_n\boxtimes P_n$ by $u_3$. 
The vertex with the largest integer $k\ge 0$ such that $(1+k, d+k)$ belongs to $G$ is denoted by $u_4$, and the vertex $(1+k, d-1+k)$ 
with the largest integer $k\ge 0$ such that $(1+k, d-1+k)$ belongs to $G$ by $u_5$. Note that these five vertices belong to $G$, except possibly $u_3$. 
Also note that, depending on $G$, $u_1$ and $u_5$ can coincide, and $u_2$ and $u_4$ can also coincide. 
Vertices $v_1, v_2, v_3, v_4$ and $v_5$ are defined similarly around $v$, as illustrated in
Figure~\ref{fig:definition}.
In $P_n\boxtimes P_n$, the edge connecting $u_3=(1,d+1)$ and $(2,d+1)$ is denoted by $e_u$.
Similarly, the edge connecting $v_3=(d+1,1)$ and $(d+1,2)$ is denoted by $e_v$.
These two edges of $P_n\boxtimes P_n$ may not belong to $G$. 
We point out that if $d=2$, then  $u_1=v_1$ and $u_5=v_5$.

\begin{figure}[!htb]
\centering
\includegraphics[scale=0.55, page=13]{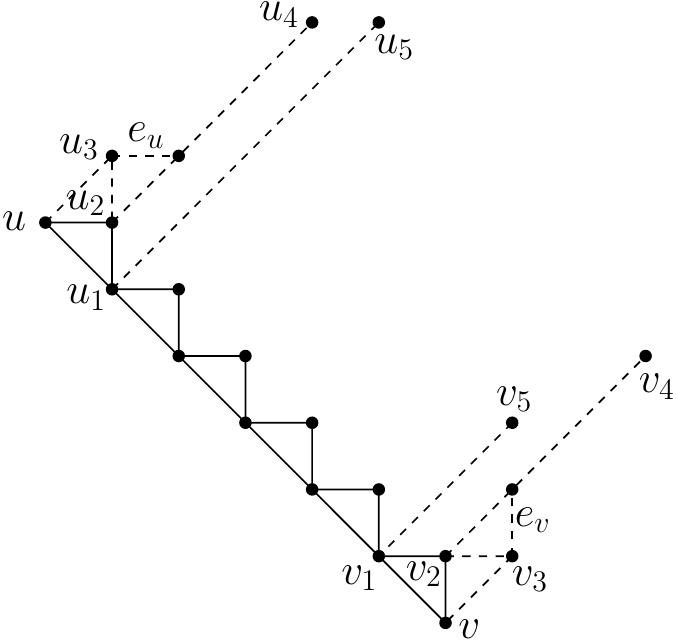}
\caption{Definition of some vertices and edges. 
} \label{fig:definition}
\end{figure}

\begin{lemma}\label{lem:AuAvx1}
Let $G$ be a MOP with $\beta(G)=2$, and let $\{u,v\}$ be a metric basis of $G$ with $d(u,v)=d\ge 2$. The following statements hold.
\begin{itemize}
\itemsep 0pt
  \item[(i)] If $\deg(u_1)=4$, then $S_1=\{u_1,u_4,v\}$ is a resolving set of $G$.
  \item[(ii)] If $\deg(v_1)=4$, then $S'_1=\{v_1,v_4,u\}$ is a resolving set of $G$.
\end{itemize}
\end{lemma}

\begin{proof} 
We will prove only $(i)$, as the proof of $(ii)$ is analogous by symmetry.
By Property~\ref{pro:p4}, any pair of vertices of $G$ is resolved by $v$, except if they have the same $y$-coordinate. Thus, we only have to show that any pair of vertices with the same $y$-coordinate is resolved by $u_1$ or $u_4$. 
Notice that if $z$ belongs to the gray region $R$ in Figure~\ref{fig:lema3}, then $d(u_1,z)=d(u,z)-1$. 
Hence, two vertices $z_1$ and $z_2$ with the same $y$-coordinate are not resolved by $u_1$, if and only if they are either one of 
the pairs $(k,d+k)$, $(1+k,d+k)$ (both vertices are at distance $k+1$ from $u_1$), 
or one of the pairs $(2+k,d+k)$ and $(3+k,d+k)$ (both vertices at distance $k+2$ from $u_1$), where $k\ge 0$, whenever these vertices belong to $G$.
In all cases, $u_4$ resolves these pairs, regardless of which edges of $P_n\boxtimes P_n$ belong to $G$ (see an illustration in Figure~\ref{fig:lema3}).
\end{proof}

    \begin{figure}[!htb]
	\centering
\includegraphics[scale=0.45, page=14]{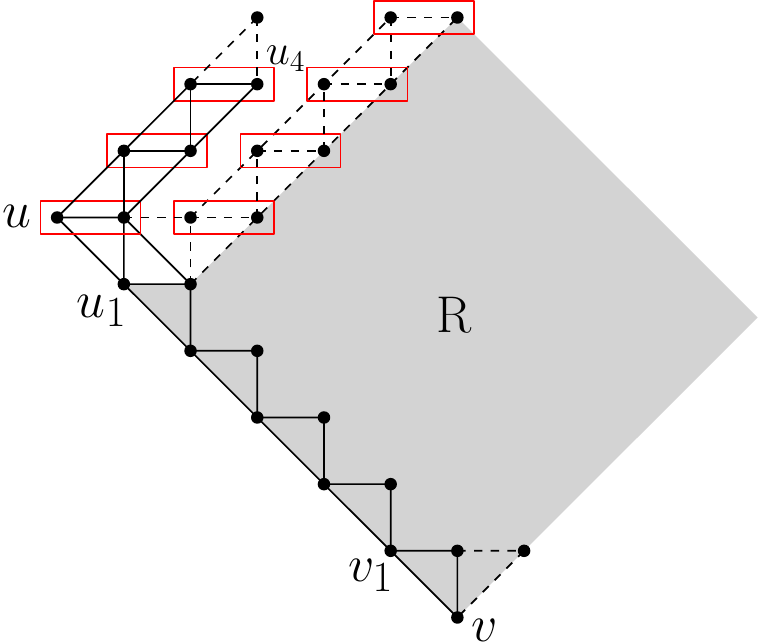}
\hspace{20mm}
\includegraphics[scale=0.45, page=15]{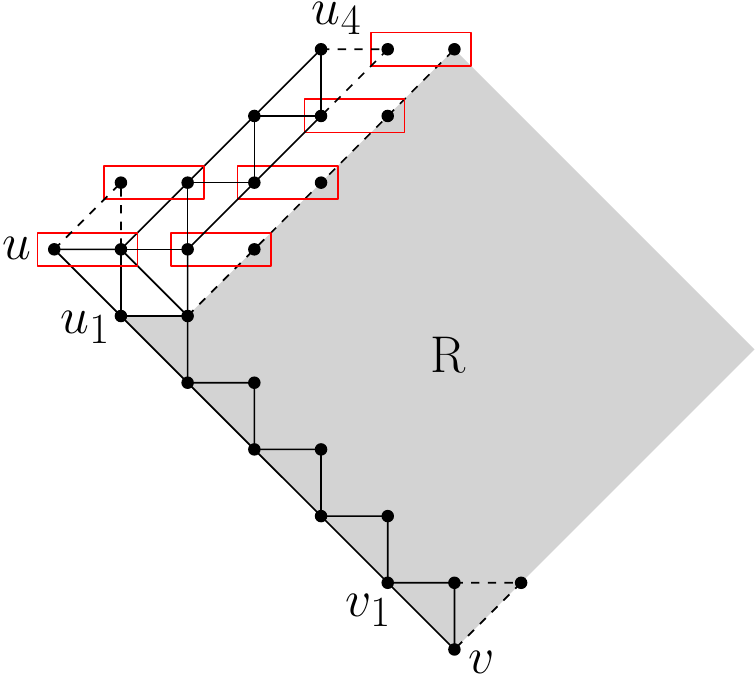}
    \caption{Illustrating Lemma~\ref{lem:AuAvx1}. The vertex $u_4$ resolves pairs of vertices not resolved by $u_1$ (in a rectangle), regardless of which zigzags belong to $G$.}
    \label{fig:lema3}
\end{figure}

\begin{lemma}\label{lem:AuAvx2}
Let $G$ be a MOP with $\beta(G)=2$, and let $\{u,v\}$ be a metric basis of $G$ with $d=d(u,v)\ge 2$. The following statements hold.
\begin{itemize}
\itemsep 0pt
  \item[(i)] If $\deg(u_1)=5$  and $e_u$ does not belong to $G$, then $S_2=\{u_1,u_5,v\}$ is a resolving set of $G$.
  \item[(ii)] If $\deg(u_1)=5$ and $e_u$ belongs to $G$, then $S_3=\{u_2,u_3,v\}$ is a resolving set of $G$.
  \item[(iii)] If $\deg(v_1)=5$ and $e_v$ does not belong to $G$, then $S'_2=\{v_1,v_5,u\}$ is a resolving set of $G$.
  \item[(iv)] If $\deg(v_1)=5$ and $e_v$ belongs to $G$, then $S'_3=\{v_2,v_3,u\}$ is a resolving set of $G$.
\end{itemize}
\end{lemma}

\begin{proof}
We only prove $(i)$ and $(ii)$, as the proofs for $(iii)$ and $(iv)$ are analogous. 
\medskip

\noindent
{\it Proof of (i).}
As in the proof of Lemma~\ref{lem:AuAvx1}, it is enough to show that any pair of vertices with the same $y$-coordinate is resolved by $u_1$ or $u_5$. 
Notice that if $z$ belongs to the  region $R$ in Figure~\ref{fig:lema4a}, then $d(u_1,z)=d(u,z)-1$.
Hence, two vertices $z_1$ and $z_2$ with the same $y$-coordinate are not resolved by $u_1$ if and only if they both belong to 
one of the sets 
$\{(0,d),(1,d),(2,d)\}$,
$\{(1,d+1),(2,d+1),(3,d+1)\}$, or $\{(1+k,d+k),(2+k,d+k)\}$, for a fixed $k\ge 2$.
Since $e_u$ does not belong to $G$, $u_5$ resolves all possible pairs (see Figure~\ref{fig:lema4a}).

\begin{figure}[!htb]
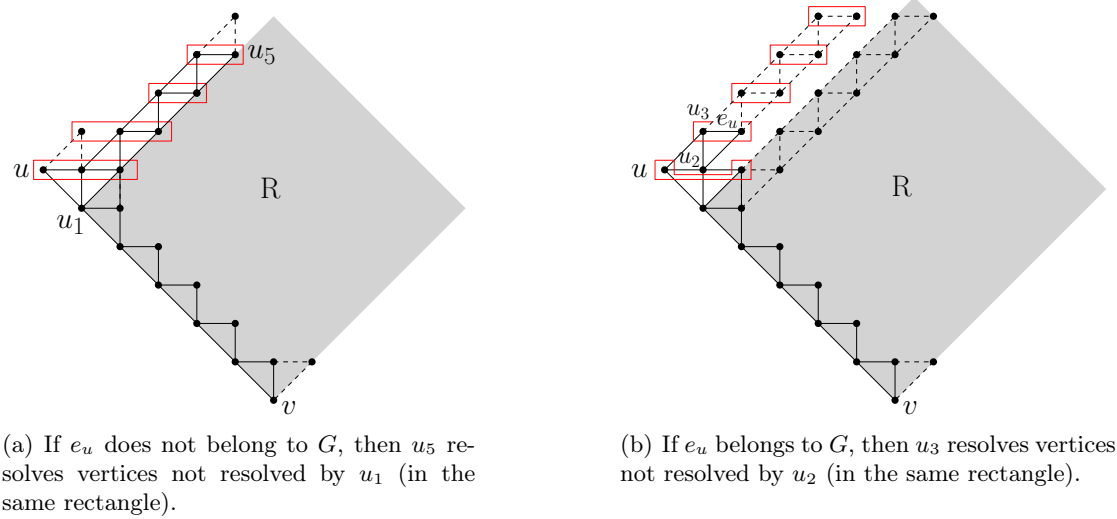

	\centering
	\subfigure[If $e_u$ does not belong to $G$, then $u_5$ resolves vertices not resolved by $u_1$ (in the same rectangle).]{
        \includegraphics[scale=0.45, page=16]{{Figures.pdf}}
		\label{fig:lema4a}
	}~~~~~~~~~~~~~~
	\subfigure[If $e_u$ belongs to $G$, then $u_3$ resolves vertices not resolved by $u_2$ (in the same rectangle).]{
        \includegraphics[scale=0.45, page=17]{{Figures.pdf}}
		\label{fig:lema4b}
	}
    \caption{Illustrating the proof of  Lemma~\ref{lem:AuAvx2}. 
    }
    \label{fig:ii)}
\end{figure}

\medskip

\noindent
{\it Proof of (ii).}
In this case, it suffices to show that any pair of vertices with the same $y$-coordinate is resolved by $u_2$ or $u_3$.
Observe that two vertices with the same $y$-coordinate belonging to the gray region $R$ in  Figure~\ref{fig:lema4b}  are resolved by $u_2$, regardless of which edges belong to $G$.
Thus, the pairs of vertices with the same $y$-coordinate not resolved by $u_2$ are 
the pair $(0,d),(2,d)$ and the pairs 
$(k,d+k)$, $(1+k,d+k)$, for a fixed $k\ge 1$ (which are at distance $k$ from $u_2$). All these pairs are resolved by $u_3$ (see an illustration in 
Figure~\ref{fig:lema4b}). 
\end{proof}

Now, we are ready to present one of the main results of this section.

\begin{theorem}\label{the:six}
Let $G$ be a MOP of order at least 7 with $\beta(G)=2$. Then, $4\le \beta'(G)\le 6$.
\end{theorem}

\begin{proof}
By Theorem~\ref{the:betap3}, $\beta'(G)\ge 4$. Let us prove that $\beta'(G)\le 6$.

Let $\{u,v\}$ be a metric basis of $G$. If $d(u,v)=1$, then $\beta'(G) = 4$ by Lemma~\ref{lem:AuAvx0}. Thus, we may assume that $d(u,v)\ge 2$. We choose $S$ as $S_1$, $S_2$, or $S_3$, as defined in Lemmas~\ref{lem:AuAvx1} and~\ref{lem:AuAvx2}, depending on whether $u_1$ has degree 4, has degree 5 and $e_u$ does not belong to $G$, or has degree 5 and $e_u$ belongs to $G$, respectively. Analogously, we choose $S'$ as $S'_1$, $S'_2$, or $S'_3$, depending on whether $v_1$ has degree 4, has degree 5 and $e_v$ does not belong to $G$, or has degree 5 and $e_v$ belongs to $G$, respectively.

\begin{figure}[!htb]
	\centering
	\subfigure[The vertex $v'\in\{v_2,v_3\}$ resolves the pairs of vertices not resolved by $u_3$ (in rectangles), regardless of which zigzags belong to $G$.]{
        \includegraphics[scale=0.40, page=18]{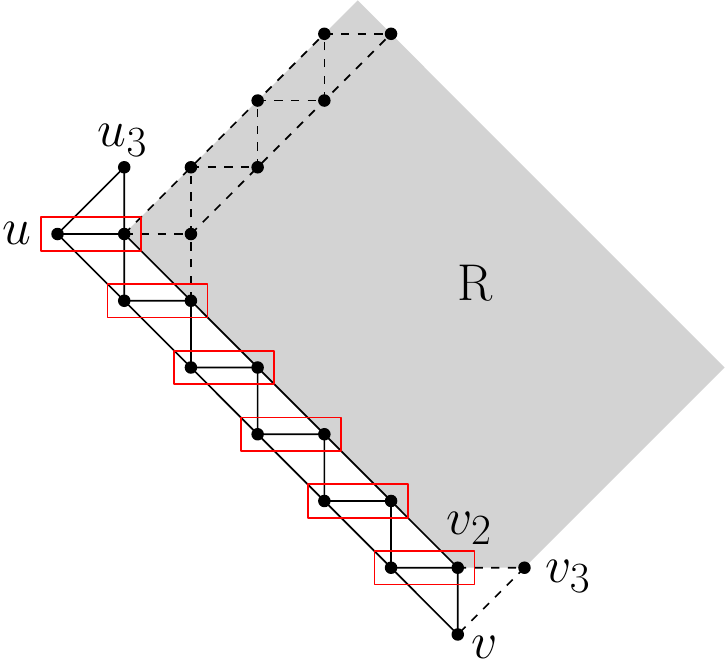}
        \includegraphics[scale=0.40, page=19]{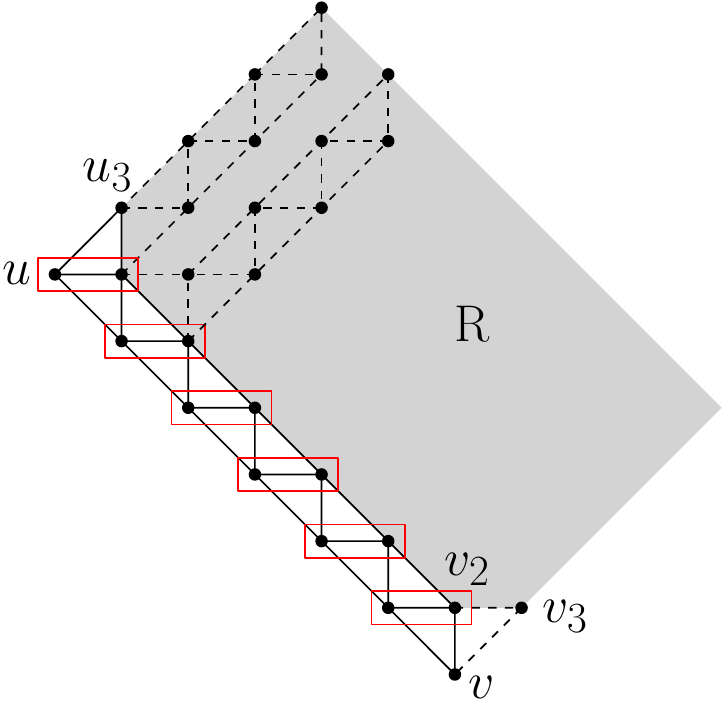}
		\label{fig:nodiamonda}
	}~
	\subfigure[The vertex $v'\in\{v_2,v_3\}$ resolves the pairs of vertices not resolved by $u_2$ (in rectangles).]{
        \includegraphics[scale=0.40, page=20]{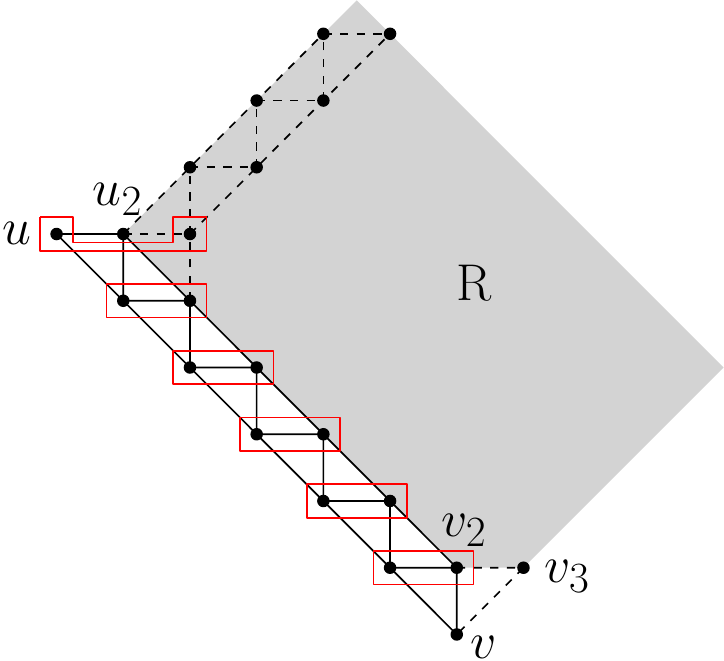}
		\label{fig:nodiamondb}
   }
    \caption{Illustrating the proof of Theorem~\ref{the:four}.}
    \label{fig:nodiamond}
\end{figure}

Let $W=S\cup S'$. Note that some of the vertices in $S$ and $S'$ may coincide, but no more than two (specifically, $u_1$ and $v_1$, and $u_5$ and $v_5$). Then, $4\le |W|\le 6$. Moreover, in the worst case, when either $u$ or $v$ is removed from $W$, the remaining vertices form a resolving set of $G$, by Lemmas~\ref{lem:AuAvx1} and~\ref{lem:AuAvx2}. Therefore, $W$ is a FTRS, which implies $4\le \beta'(G)\le 6$.
\end{proof}

The bounds of Theorem~\ref{the:six} can be tightened according to $r(G)$. Recall that $r(G)$ is the number of vertices of degree 5 on the shortest path between the vertices of a given metric basis. 

\begin{theorem}\label{the:four}
Let $G$ be a MOP of order $n\ge7$ with $\beta(G)=2$, and let $B=\{u,v\}$ be a metric basis of $G$. If $r(G)=0$, then $\beta'(G)=4$.
\end{theorem}

\begin{proof}
By Theorem~\ref{the:six}, it is enough to show that $G$ has a FTMB with four vertices.  
If $d(u,v)=1$, then $r(G)=0$ and $\beta'(G)=4$ by Lemma~\ref{lem:AuAvx0}. Then, we may assume $d(u,v)=d\ge 2$.

Consider the vertices $u_2, u_3, v_2$ and $v_3$ of $P_n\boxtimes P_n$. 
If $u_3$ belongs to $G$, we define $u'$ as $u_3$; otherwise, we define $u'$ as $u_2$. Similarly, if $v_3$ belongs to $G$, we define $v'=v_3$; otherwise, we define $v'=v_2$. 
We prove that the set $W=\{u,u',v,v'\}$ is a FTRS of $G$. 
Obviously, the set obtained after removing $u'$ or $v'$ remains a resolving set because it contains a metric basis. 
Hence, we have to prove that $W\setminus \{u\}$ and $W\setminus \{v\}$ are resolving sets. By symmetry, it suffices to prove that $W\setminus \{u\}$ is a resolving set. 
Thus, by Property~\ref{pro:p4}, we have to show that any pair of vertices with the same $y$-coordinate is resolved by $u'$ or $v'$.

Suppose first that $u'=u_3$ (see Figure~\ref{fig:nodiamonda} for an illustration). 
It is easy to check that two vertices with the same $y$-coordinate in the region $R$ of Figure~\ref{fig:nodiamonda} are resolved by $u_3$, regardless of which vertices and edges belong to $R$, because the distance to $u_3$ strictly increases when the $x$-coordinate increases. 
Hence, it can easily be derived that the only pairs of vertices with the same $y$-coordinate not resolved by $u_3$ are the pairs of vertices of the spine with the same $y$-coordinate, that is, the pairs
$(0+k,d-k)$ and $(1+k,d-k)$, for $k\in \{0,\dots ,d-1\}$. These pairs are resolved by $v'$.

Suppose now that $u'=u_2$ (see Figure~\ref{fig:nodiamondb} for an illustration). 
In this case, the pairs of vertices with the same $y$-coordinate not resolved by $u_2$ are $(0,d), (2,d)$, if this last vertex belongs to $G$, and the pairs of vertices of the spine with the same $y$-coordinate, distinct from $d$. These pairs are resolved by $v'$.
\end{proof}

\begin{figure}[!htb]
	\centering
	\subfigure[The vertex $w$ resolves the pairs of vertices not resolved by $u'=u_3$ regardless of which zigzags belong to $G$.]{
        \includegraphics[scale=0.4, page=21]{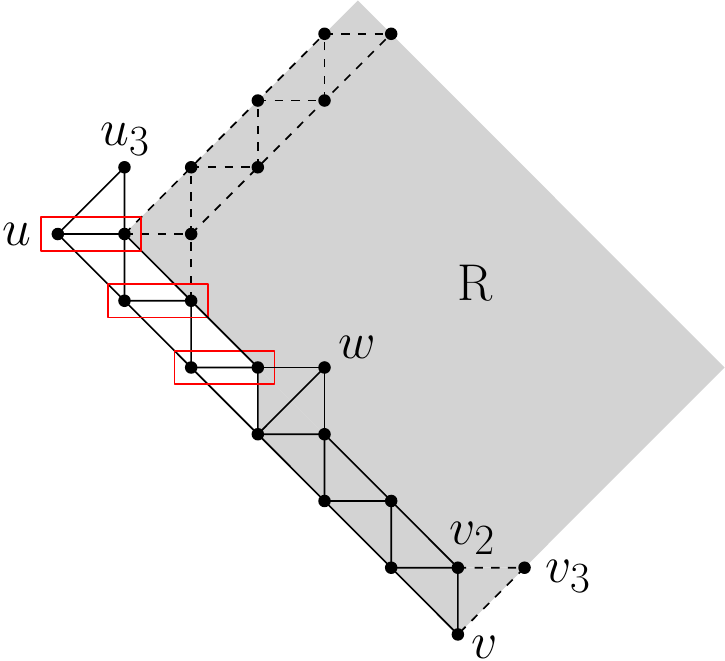}
        \includegraphics[scale=0.4, page=22]{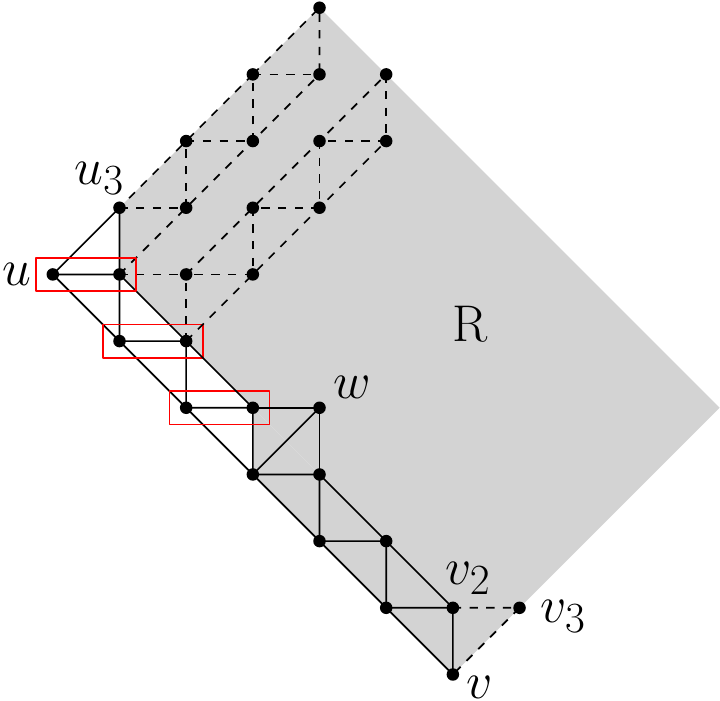}
		\label{fig:onediamonda}
	}~~
	\subfigure[The vertex $w$ resolves the pairs of vertices not resolved by $u'=u_2$.]{
        \includegraphics[scale=0.4, page=23]{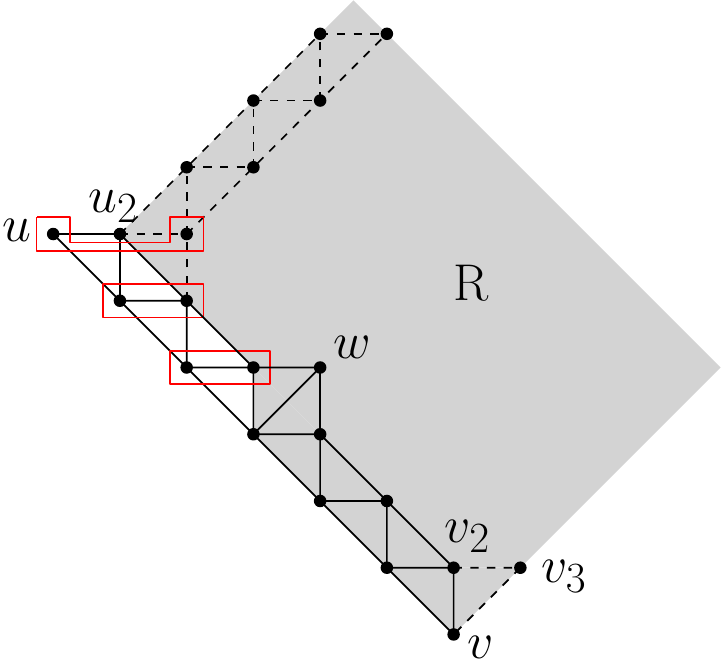}
		\label{fig:onediamondb}
   }
    \caption{Illustrating the proof of Theorem~\ref{the:five}.
    A pair of vertices with the same $y$-coordinate in the gray region $R$ is resolved by $u'$.}
    \label{fig:onediamond}
\end{figure}

\begin{theorem}\label{the:five}
Let $G$ be a MOP of order at least $7$ with $\beta(G)=2$, and let $B=\{u,v\}$ be a metric basis of $G$. If $r(G)=1$, then $4\le \beta'(G)\le 5$.
\end{theorem}

\begin{proof}
Notice first that $d(u,v)=d\ge 2$ because $r(G)=1$. Let $(i,d-i)$ be the vertex of degree 5 in the shortest path between $u$ and $v$, with $1\le i \le d-1$. Let $w$ be the vertex $(i+1, d-i+1)$ and let $y_w=d-i+1$ be its $y$-coordinate; see Figure~\ref{fig:onediamond}. Let $u'$ and $v'$ be as defined in the proof of Theorem~\ref{the:four}, that is, $u'=u_3$ ($v'=v_3$) if $G$ contains $u_3$ ($v_3$), and $u'=u_2$ ($v'=v_2$) otherwise.
We claim that the set $W=\{u,u',v,v',w\}$ is a FTRS of $G$. Since $\{u,v\}$ is a resolving set, by symmetry, it is enough to prove that $W\setminus \{u\}$ is a resolving set. Furthermore, by Property~\ref{pro:p4}, it suffices to show that two vertices with the same $y$-coordinate are resolved by a vertex in $\{u',v',w\}$.

Suppose first that $u'=u_3$; see Figure~\ref{fig:onediamonda}. The vertices with the same $y$-coordinate not resolved by $u'$ are the pairs of vertices of the spine with the same $y$-coordinate between $y_w$ and $d$,
and all of them are resolved by $w$. 
Now, if $u'=u_2$ (see Figure~\ref{fig:onediamondb}),
the pairs of vertices with the same $y$-coordinate  not resolved by $u'$ are $(0,d), (2,d)$, if this last vertex belongs to $G$, and the pairs of vertices of the spine with the same $y$-coordinate between $y_w$ and $d-1$. 
In this case, $w$ resolves all these pairs.
\end{proof}

To tighten the bounds when $r(G)\ge 4$, we need to introduce some terminology (see Figure~\ref{fig:AuAv}).
If $w=(i,d-i)$  is a vertex that belongs to the shortest path from $u$ to $v$, distinct from $u$ and $v$, we define $A_u(w)$ as the set of vertices in the component of $G- \{w,w'\}$ that contains $u$, where $w'=(i+1,d-i)$, and $A'_u(w)$ as the subset of $A_u(w)$ formed by the vertices with $x$-coordinate at most $i$. 
Similarly, we define $A_v(w)$ as the set of vertices in the  component of $G - \{w,w''\}$ that contains $v$, where $w''=(i,d-i+1)$, and $A_v'(w)$ as the subset of $A_v(w)$ formed by the vertices with $y$-coordinate at most $d-i$.

\begin{figure}[!htb]
	\centering	
	\subfigure[$A_u(w)$ when $e$ belongs to $G$.]{
        \includegraphics[scale=0.45, page=24]{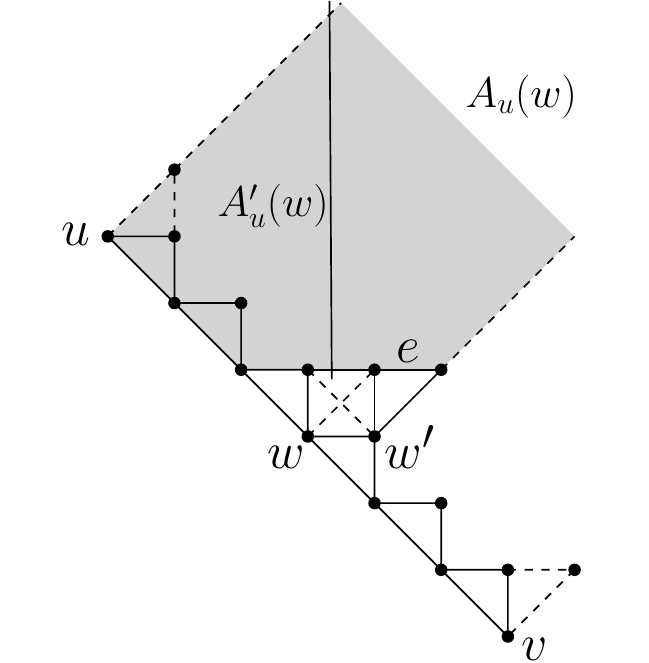}
		\label{fig:AuAv_1}
	}~~~~~~~~~~~~~
	\subfigure[$A_u(w)$ when $e$ does not belong to $G$.]{
        \includegraphics[scale=0.45, page=25]{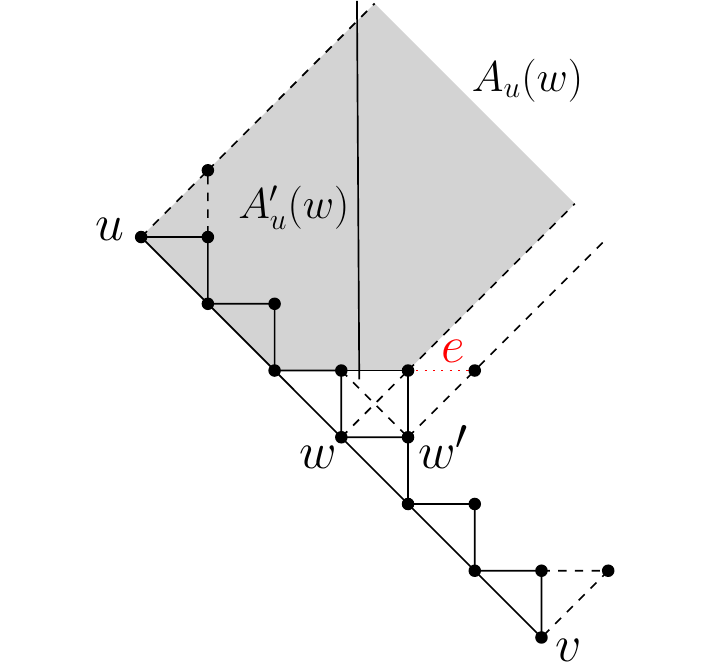}
		\label{fig:AuAv_2}
	}
    \caption{Vertices of the sets $A_u(w)$ and $A'_u(w)$. 
    In the gray region, the vertices of $A_u(w)$ are shown. Those at the left of the vertical line also belong to $A'_u(w)$. 
    }   \label{fig:AuAv}
\end{figure}

\begin{lemma}\label{lem:AuAv}
Let $G$ be a MOP of order at least 7 with $\beta(G)=2$, and let $\{u,v\}$ be a metric basis of $G$.
Let $w=(i,d-i)$ be a vertex of degree $5$ in the shortest path between $u$ and $v$. If $z\notin A_u(w)$, then $z$ does not resolve any pair of vertices in $A_u'(w)$ that have the same $y$-coordinate.
If $z\notin A_v(w)$, then $z$ does not resolve any pair of vertices in $A_v'(w)$ having the same $x$-coordinate.
\end{lemma}
\begin{proof}

Since $\deg(w)=5$, observe that given a vertex $z\notin A_u(w)$ and a vertex $t\in A_u'(w)$, there is always a shortest path from $z$ to $t$ passing through $w$. In particular, this implies $t_y=d(v,t)=d(v,w)+d(w,t) = (d-i)+d(w,t)$, where $t_y$ is the $y$-coordinate of $t$, and $d(z,t)= d(z,w)+d(w,t) = d(z,w)+t_y-(d-i)$. Hence, the distance from $z$ to the vertices in $A_u'(w)$ having the same $ y$-coordinate is the same. Therefore, $z$ does not resolve a pair of vertices in $A_u'(w)$ with the same $y$-coordinate. 

By symmetry, we derive that if $z\notin A_v(w)$, then $z$ does not resolve any pair of vertices in $A_v'(w)$ having the same $x$-coordinate.
\end{proof}

Given a metric basis $B=\{u,v\}$ of $G$, let $w_1,\dots, w_{r(G)}$ be the set vertices of degree $5$ in the shortest path from $u$ to $v$, where $w_j=(i_j,d-i_j)$, for $j\in \{1,\dots ,r(G)\}$, and $0<i_1<\dots<i_{r(G)}<d=d(u,v)$. 
Then, the following lemma holds.

\begin{lemma}\label{lem:Au23}
Let $B=\{u,v\}$ be a metric basis of a MOP $G$, and let $W$ be a FTRS of $G$. If $r(G)\ge  1$, then $W$ has at least two vertices in $A_u(w_1)$ and at least two vertices in $A_v(w_{r(G)})$.
Moreover, if $r(G)\ge  3$, then $W$ has at least three vertices in $A_u(w_3)$ and at least three vertices in $A_v(w_{r(G)-2})$.
\end{lemma}
\begin{proof} 
Assume first that $r(G)\ge  1$. 
By Lemma~\ref{lem:AuAv}, the vertices not belonging to $A_u(w_1)$ do not resolve the pair $u$, $u_2$.
Hence, $W$ has at least one vertex $t$ in $A_u(w_1)$. 
But $W\setminus \{t\}$ must be a resolving set. Thus, again by Lemma~\ref{lem:AuAv},
$W\setminus \{t\}$ has at least one vertex in $A_u(w_1)$. Therefore, $W$ has at least two vertices in $A_u(w_1)$.
By symmetry, $W$ has at least two vertices in $A_v(w_{r(G)})$. 

Assume now $r(G)\ge  3$.
Suppose that $W$ has at most two vertices $\{t_1,t_2\}$ in $A_u(w_3)$. We have proven that $W$ has at least two vertices in $A_u(w_1)$. Hence, since $A_u(w_1)\subseteq A_u(w_3)$, we derive $W\cap A_u(w_1)=W\cap A_u(w_3)=\{t_1,t_2\}$. 
As $W$ is a FTRS, $W\setminus \{t_1\}$ and $W\setminus \{t_2\}$ are resolving sets. Thus, by Lemma~\ref{lem:AuAv}, both vertices $t_1$ and $t_2$ must resolve every pair of vertices in $A'_u(w_3)$ with the same $y$-coordinate. 
However, as we next prove, $u$ is the only vertex in $A_u(w_1)$ that resolves all these pairs.
More precisely, for each $t\in A_u(w_1)\setminus \{u\}$, 
we will show that there is a pair of vertices in $A_u'(w_3)$ with the same $y$-coordinate not resolved by $t$.
We distinguish three cases depending on the subgraph $H_i$ to which $t$ belongs. Recall that $H_i$ is the subgraph attached to $(i,d-i)$ in the embedding of $G$; see Figure~\ref{fig:embedded}.

\begin{figure}[!htb]
\centering
	\subfigure[]{
        \includegraphics[scale=0.9, page=26]{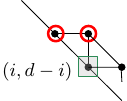}
		\label{fig:casHi_1}
	}~~~
    	\subfigure[]{
        \includegraphics[scale=0.9, page=27]{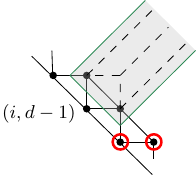}
		\label{fig:casHi_2}
	}~~~
    	\subfigure[]{
        \includegraphics[scale=0.9, page=28]{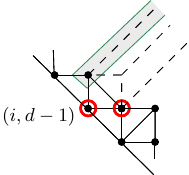}
		\label{fig:casHi_3}
	}~~~
    	\subfigure[]{
        \includegraphics[scale=0.9, page=29]{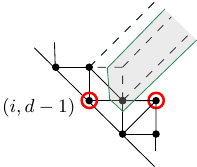}
		\label{fig:casHi_4}
	}
    \caption{
    Vertices in the gray region do not resolve the pair of circled vertices.}
    \label{fig:casHi}
\end{figure}
\bigskip

\noindent
{\bf Case 1}:  $t\in V(H_i)$, $1\le i<i_1$ (see Figure~\ref{fig:casHi}).
\medskip

    If $t=(i,d-i)$, then $t$ does not resolve the pair $(i-1,d-i+1)$ and $(i,d-i+1)$; see Figure~\ref{fig:casHi_1}. Thus, we can suppose $t\not =(i,d-i)$ in the next reasoning.
    
    If $\deg ((i+1,d-i-1))=4$, then $t$ does not resolve the pair $(i+1,d-i-1)$ and $(i+2,d-i-1)$; see Figure~\ref{fig:casHi_2}.
    
    If $\deg ((i+1,d-i-1))=5$ and $t=(i+k,d-i+1+k)$, for some $k\ge 0$, then $t$ does not resolve  the pair $(i,d-i)$ and $(i+1,d-i)$; see Figure~\ref{fig:casHi_3}. 
    
    Finally, if $\deg ((i+1,d-i-1))=5$ and $t=(i+1+k,d-i+k)$ or $t=(i+1+k,d-i+1+k)$, for some $k\ge 0$, then $t$ does not resolve  the pair $(i,d-i)$ and $(i+2,d-i)$; see Figure~\ref{fig:casHi_4}.

\begin{figure}[!htb]
\centering
	\subfigure[]{
        \includegraphics[scale=0.8, page=30]{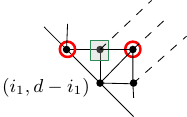}
		\label{fig:casHi1_1}
	}
    	\subfigure[]{
        \includegraphics[scale=0.8, page=31]{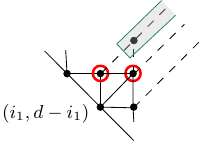}
		\label{fig:casHi1_2}
	}
    	\subfigure[]{
        \includegraphics[scale=0.8, page=32]{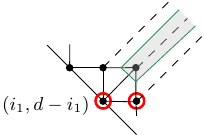}
		\label{fig:casHi1_3}
	}
        \subfigure[]{
        \includegraphics[scale=0.8, page=33]{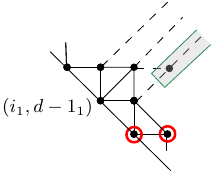}
		\label{fig:casHi1_4}
	}
    	\subfigure[]{
        \includegraphics[scale=0.8, page=34]{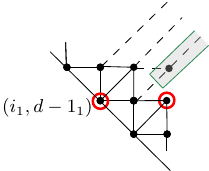}
		\label{fig:casHi1_5}
	}
    \caption{
    Vertices in the gray region do not resolve the pair of circled vertices.}
    \label{fig:casHi1}
\end{figure}
\medskip

\noindent
{\bf Case 2}:  $t\in V(H_{i_1})$  (see Figure~\ref{fig:casHi1}).
\medskip

If $t=(i_1,d-i_1+1)$, then $t$ does not resolve the pair $(i_1-1,d-i_1+1)$ and $(i_1+1,d-i_1+1)$ see Figure~\ref{fig:casHi1_1}. 

If $t=(i_1+k,d-i_1+1+k)$, for some $k\ge 1$, then $t$ does not resolve the pair $(i_1,d-i_1+1)$ and $(i_1+1,d-i_1+1)$; see Figure~\ref{fig:casHi1_2}. 

If $t=(i_1+1+k,d-i_1+1+k)$, for some $k\ge 0$, then $t$ does not resolve the pair $(i_1,d-i_1)$ and $(i_1+1,d-i_1)$; see Figure~\ref{fig:casHi1_3}. 

If $t=(i_1+1+k,d-i_1+k)$, for some $k\ge 1$, 
and $\deg ((i_1+1,d-i_1-1))=4$, then $t$ does not resolve the pair
$(i_1+1,d-i_1-1)$ and $(i_1+2,d-i_1-1)$; see Figure~\ref{fig:casHi1_4}. 

Finally, if $t=(i_1+1+k,d-i_1+k)$, for some $k\ge 1$, and $\deg ((i_1+1,d-i_1-1))=5$, then $t$ does not resolve the pair $(i_1,d-i_1)$ and $(i_1+2,d-i_1)$; see Figure~\ref{fig:casHi1_5}.

\begin{figure}[!htb]
\centering
	\subfigure[]{
        \includegraphics[scale=0.9, page=35]{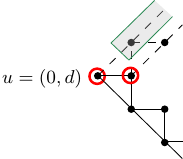}
		\label{fig:casH0_1}
	}~~~~~~
    	\subfigure[]{
        \includegraphics[scale=0.9, page=36]{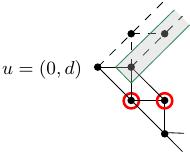}
		\label{fig:casH0_2}
	}~~~~~~
    	\subfigure[]{
        \includegraphics[scale=0.9, page=37]{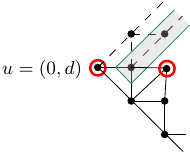}
		\label{fig:casH0_3}
	}
    \caption{
    Vertices in the gray region do not resolve the pair of circled vertices.}
    \label{fig:casH0}
\end{figure}
\medskip

\noindent
{\bf Case 3}:  $t\in V(H_0)$ (see Figure~\ref{fig:casH0}).
\medskip

If $t=(0+k,d+k)$, for some $k\ge 1$, then $t$ does not resolve the pair $u=(0,d)$ and $u_2=(1,d)$; see Figure~\ref{fig:casH0_1}.

If $t=(1+k,d+k)$, for some $k\ge 0$, and $\deg ((1,d-1))=4$, then $t$ does not resolve the pair $(1,d-1)$ and $(2,d-1)$; see Figure~\ref{fig:casH0_2}.

Finally, if $t=(1+k,d+k)$, for some $k\ge 0$, and $\deg ((1,d-1))=5$, then $t$ does not resolve the pair $u=(0,d)$ and $(2,d)$; see Figure~\ref{fig:casH0_3}.
\medskip

Therefore, $u$ is the only vertex in $A_u(w_1)$ that resolves all pairs of vertices in $A_u'(w_3)$, as claimed. As a consequence, $W$ must have at least three vertices in $A_u(w_3)$. Moreover, by symmetry, $W$ must also contain at least three vertices in $A_v(w_{r(G)-2})$.
\end{proof}

\begin{theorem}
\label{thm:r4r6}
Let $G$ be a MOP of order at least $7$ with $\beta(G)=2$, and let $B=\{u,v\}$ be a metric basis of $G$. If $r(G)\ge 4$, then $\beta'(G)\ge 5$, and if $r(G)\ge 6$,  then $\beta'(G)= 6$. 
\end{theorem}
\begin{proof}
Let $W$ be a FTRS of $G$ of minimum size. If $r(G)\ge  4$, then $A_u(w_3)$ and $A_v(w_{r(G)})$ are disjoint sets. By Lemma~\ref{lem:Au23}, $W$ has at least 3 vertices in $A_u(w_3)$ and at least 2 vertices in $A_v(w_{r(G)})$. Hence, $\beta'(G)=|W|\ge 5$.

If $r(G)\ge  6$, then $A_u(w_3)$ and $A_v(w_{r(G)-2})$ are disjoint sets. By Lemma~\ref{lem:Au23}, $W$ has at least 3 vertices in $A_u(w_3)$ and at least 3 vertices in $A_v(w_{r(G)-2})$. Thus, $\beta'(G)=|W|\ge 6$. By Theorem~\ref{the:six}, $\beta'(G)=6$.
\end{proof}

\begin{figure}[!htb]
\centering
\includegraphics[scale=0.8, page=38]{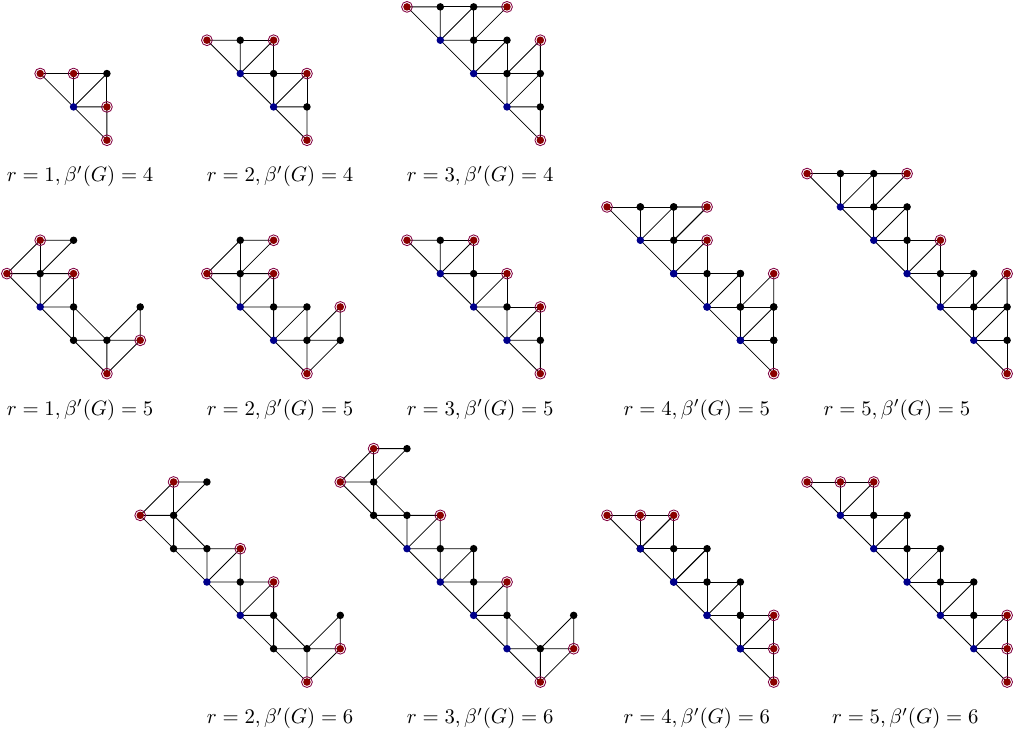}
    \caption{Examples of MOPs $G$ with $\beta(G)=2$ and $\beta'(G) \in\{4,5,6\}$, where $r$ denotes $r(G)$.
}
    \label{fig:tightbounds}
\end{figure}

We conclude this section by summarizing the results of Theorems~\ref{the:six}, \ref{the:four}, \ref{the:five} and \ref{thm:r4r6}, and showing that all the values between the given lower and upper bounds are realizable.

\begin{theorem} If $G$ is a MOP of order at least 7 with $\beta(G)=2$, then $4\le \beta'(G)\le 6$.
Moreover, for a metric basis $B$ of $G$, 
\begin{enumerate}[label=(\roman*), leftmargin=*] 
\item if $r(G)=0$, then $\beta'(G)=4$;
\item if $r(G)=1$, then $4\le \beta'(G)\le 5$;
\item if $r(G)\in \{ 2,3\}$, then $4\le \beta'(G)\le 6$;
\item if $r(G)\in \{4,5\}$, then $5\le \beta'(G)\le 6$;  
\item if $r(G)\ge 6$, then $\beta'(G)= 6$.
\end{enumerate}
Moreover, all values between the given lower and upper bounds are realizable. 
\end{theorem}
\begin{proof} 
It is enough to show an example for each combination of possible values of $r(G)$ and $\beta'(G)$ in the statement. 
An example for each case is depicted in  Figure~\ref{fig:tightbounds}, where $r$ denotes $r(G)$ and red circled vertices form a FTMB.
The FTMD was calculated by computer.
\end{proof}

\section{An upper bound on the FTMD of MOPs}\label{sec:upper}

The main goal of this section is to show that any MOP $G$ of order  $n\ge5 $ has a FTRS $W$ of size $\cn$, which can be built in linear time. Furthermore, this set $W$ is {\em dense}, meaning that for every vertex not in $W$, its two neighbors on the boundary of $G$ belong to $W$. To prove this result, we will need the following lemma by 
Shermer~\cite{Sher91}, which shows that a MOP can be split into two MOPs of specific sizes.

\begin{lemma}\cite{Sher91}\label{lem:orourke}   
Let $l\ge 2$ be a positive integer and let $G$ be a MOP of order $n$, with $n\ge 2l$. Then, there is a diagonal that splits $G$ into two MOPs, one of them of order $m$, with $m\in \{l+1, \ldots, 2l-1 \}$. 
\end{lemma}

\begin{theorem}
If $G$ is a MOP of order $n\ge 5$, then $\beta'(G)\le \cn$. Moreover, a FTRS of cardinality $\cn$ can be built in linear time. 
\end{theorem}
\begin{proof}
The proof proceeds by induction on $n$, with the base cases being $n \in \{5, 6, 7\}$. 
Figure~\ref{fig:dense7} shows a dense FTRS $W$ of size $\cn$, for every MOP $G$ of order $n \in \{5, 6, 7\}$. 

\begin{figure}[!htb]
\centering
\includegraphics[scale=0.75, page=39]{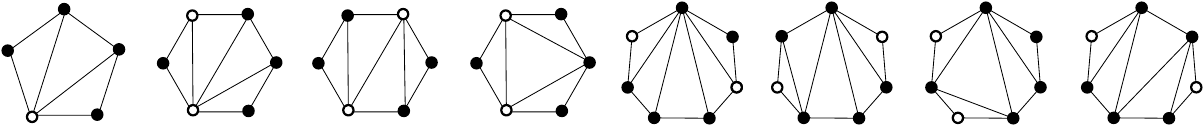}
        \caption{All MOPs of order 5, 6, and 7. Black vertices in each MOP form a  dense FTRS of size $\cn$.}
    \label{fig:dense7}
\end{figure}

Now, assume $G$ is a MOP of order $n\ge 8$, and the theorem is true for MOPs of order at most $n-1$. Taking $l=4$  in Lemma~\ref{lem:orourke}, there is a diagonal that splits $G$ into two MOPs, one of them, say $G_0$,  
of order $n_0\in \{ 5,6,7\}$.
Thus, $G$ is one of the MOPs of Figure~\ref{fig:cota_todos},
where one of the cases 1(i), 1(ii) or 1(iii) holds if $n_0=5$; one of the cases 2(i) or 2(ii) holds if $n_0=6$ and the preceding case does not hold; and one of the cases 3(i), 3(ii) or 3(iii) holds if $n_0=7$ and the preceding cases do not hold. In each of the eight configurations, we show how to build a dense FTRS $W$ of size $|W|=\cn$.

\begin{figure}[!htb]
\centering
\includegraphics[scale=0.75, page=40]{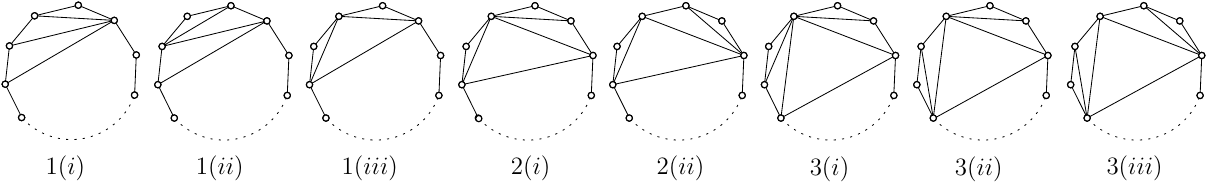}
    \caption{If $G$ is a MOP of order at least 8, one of these cases (or its symmetric) holds.}
    \label{fig:cota_todos}
\end{figure}
\bigskip

\noindent
{\bf Cases $1(i)$, $1(ii)$, $1(iii)$.} 
\medskip

Suppose that the vertices of $G$ are labeled as shown in Figure~\ref{fig:casodiag5}. Let $W'$ be a dense FTRS of $G'=G- \{a,b,c\}$ of size $|W'| = \cnp \ge 4$, which exists by the inductive hypothesis because $G'$ is a MOP of order $n-3\ge 5$. Let $W=W'\cup \{a,c\}$. In each of Cases $1(i), 1(ii)$ and $1(iii)$, we prove that $W$ is a dense FTRS of $G$ of size $\cn$. 

\begin{figure}[!htb]
\centering
	\subfigure[Case $1(i)$.]{
        \includegraphics[scale=0.8, page=41]{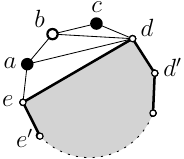}
		\label{fig:casodiag5_1}
	}~~~~~~
    	\subfigure[Case $1(ii)$.]{
        \includegraphics[scale=0.8, page=42]{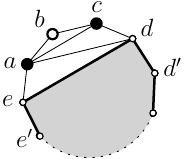}
		\label{fig:casodiag5_2}
	}~~~~~~
    	\subfigure[Case $1(iii)$.]{
        \includegraphics[scale=0.8, page=43]{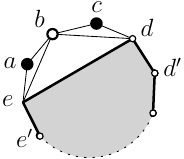}
		\label{fig:casodiag5_3}
	}
    \caption{
    Cases $1(i)$, $1(ii)$ and $1(iii)$. 
    In gray, the MOP $G'=G- \{a,b,c\}$. Black vertices belong to $W$.}
    \label{fig:casodiag5}
\end{figure}

We first show that $W$ 
is a FTRS of $G$.
Since the distance between two vertices of $G'$ is the same in $G$ and in $G'$, it only remains to check that every pair of vertices of $G$, with at least one of them in $\{a,b,c\}$, is resolved by two vertices of $W$. 
Denote by $u$ a vertex not in $\{a,b,c,d,e\}$, and let $u'$ and $u''$ be its neighbors on the boundary of $G'$.
As $W'$ is a dense FTRS of $G'$, 
at least one of the vertices  $d$ or $d'$ belongs to $W'$, at least one of the vertices  $e$ or $e'$ belongs to $W'$ and, for $u\notin \{a,b,c,d,e\}$, either $u\in  W'$, or $u',u''\in W'$.

Table~\ref{tab:1i} analyzes Case $1(i)$, giving for every pair $x,y$ of vertices with $x\in \{a,b,c\}$, one or more pairs of vertices so that at least one of these pairs is formed by two vertices of $W$ that resolve the pair $x,y$.  For example, the pair $a,e$ is resolved by $a$ and $e$ if $e\in W$, or by $a$ and $e'$ if $e'\in W$. Recall that at least one of $e$ and $e'$ belongs to $W'$. Note that the pair $b, d$ is resolved by two vertices of $W'$ because $d(w',b)=d(w',d)+1$ for any vertex $w'\in W'$. As a consequence, $W$ is a FTRS of $G$ in Case $1(i)$. Using Tables~\ref{tab:1ii} and~\ref{tab:1iii}, it can also be checked that $W$ is a FTRS of $G$ in Cases $1(ii)$ and $1(iii)$, respectively.

\begin{table}[!htb]
  \centering 
  \begin{tabular}{r|c|c|c|c|c}
            $x\diagdown$\raisebox{1mm}{$y$}& $b$ & $c$ & $d$ & $e$ & $u$
            \\
            \hline
            $a$ & $a,c$ & $a,c$ & 
            $a,c$&
            $a,e$ or $a,e'$
            & $a,u$ or $a,u'$ or $a,u''$
            \\
            \hline
            $b$& -- & $a,c$ & 2 vertices of $W'$& 
            $c,e$ or $c,e'$
            & $a,c$
            \\
            \hline
            $c$& --&--& $a,c$& $a,c$&
            $c,u$ or $c,u'$ or $c,u''$
            \\
            \hline
        \end{tabular}
  \caption{Case $1(i)$.}
  \label{tab:1i}
\end{table}

\begin{table}[!htb]
  \centering 
  \begin{tabular}{r|c|c|c|c|c}
            $x\diagdown$\raisebox{1mm}{$y$}& $b$ & $c$ & $d$ & $e$ & $u$
            \\
            \hline
            $a$ & $a,e$ or $a,e'$ & $a,c$ & 
            $a,d$ or $a,d'$&
            $a,c$
            &$a,c$ 
            \\
            \hline
            $b$& -- & $c,d$ or $c,d'$ & 2 vertices of $W'$& $c,e$ or $c,e'$
            & $a,c$
            \\
            \hline
            $c$& --&--& $c,d$ or $c,d'$&
            $c,e$ or $c,e'$ & $a,c$
            \\
            \hline
        \end{tabular}
  \caption{Case $1(ii)$.}
  \label{tab:1ii}
\end{table}

\begin{table}[!htb]
  \centering 
  \begin{tabular}{r|c|c|c|c|c}
            $x\diagdown$\raisebox{1mm}{$y$}& $b$ & $c$ & $d$ & $e$ & $u$
            \\
            \hline
            $a$ & $a,c$  & $a,c$ & 
            $a,c$&
            $a,e$ or $a,e'$
            &$a,u$ or $a,u'$ or $a,u''$
            \\
            \hline
            $b$& -- & $a,c$  & $a,d$ or $a,d'$ & $c,e$ or $c,e'$
            & $a,c$
            \\
            \hline
            $c$& --&--& $c,d$ or $c,d'$& $a,c$
             & $c,u$ or $c,u'$ or $c,u''$
            \\
            \hline
        \end{tabular}
  \caption{Case $1(iii)$.}
  \label{tab:1iii}
\end{table}

Finally, observe that since $W'$ is dense and $a,c\in W$, then $W$ is also dense. 
Morover, $|W|=|W'|+2 = \cnp+2=\cn$. Therefore,  $W$ is a dense FTRS of size $\cn$ in each of Cases $1(i), 1(ii)$ and $1(iii)$.

\begin{figure}[!htb]
\centering
	\subfigure[$2(i)$; $b\in W'$.]{
        \includegraphics[scale=0.8, page=44]{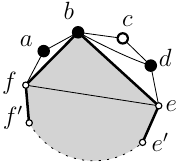}
		\label{fig:casodiag6_1}
	}~~~~
    	\subfigure[$2(i)$; $b\not\in W'$.]{
        \includegraphics[scale=0.8, page=45]{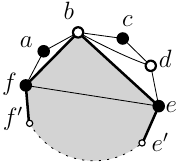}
		\label{fig:casodiag6_2}
	}~~~~~~~~~~~~
    	\subfigure[$2(ii)$; $b\in W'$.]{
        \includegraphics[scale=0.8, page=46]{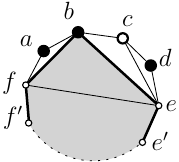}
		\label{fig:casodiag6_3}
    }~~~~
    	\subfigure[$2(ii)$; $b\not\in W'$.]{
        \includegraphics[scale=0.80, page=47]{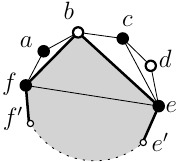}
		\label{fig:casodiag6_4}
	}
    \caption{
    Cases $2(i)$ and $2(ii)$. In gray, the MOP  $G'=G -\{a,c,d\}$.  Black vertices belong to $W$.}
    \label{fig:casodiag6}
\end{figure}

\bigskip

\noindent
{\bf Cases $2(i)$, $2(ii)$.} 
\medskip

Suppose that the vertices are labeled as shown in Figure~\ref{fig:casodiag6}. 
Let $W'$ be a dense FTRS of $G'=G- \{a,c,d\}$ of size $|W'|=\cnp\ge 4$, 
which exists by the inductive hypothesis.
Denote by $u$ a vertex not in $\{a,b,c,d,e,f\}$, and let $u'$ and $u''$ be its neighbors on the boundary of $G'$. 
We distinguish cases according to whether 
$b$ belongs to $W'$ or not.
\medskip

\noindent
$\bullet$ 
 $b\in W'$. Let $W=W'\cup \{a,d\}$; see 
Figures~\ref{fig:casodiag6_1} and~\ref{fig:casodiag6_3}. As in the previous cases, we only need to check that every pair of vertices with at least one of them in $\{a, c, d\}$ is resolved by two vertices of $W$. Tables~\ref{tab:2ibsi} and~\ref{tab:2iibsi} show the pairs of vertices that resolve the pair $x,y$, for
$ x\in \{a, c, d\}$,  
in each of Cases $2(i)$ and $2(ii)$, respectively.
\medskip

\begin{table}[!htb]
  \centering 
  \begin{tabular}{r|c|c|c|c|c|c}
            $x\diagdown$\raisebox{1mm}{$y$}& $b$ & $c$ & $d$ & $e$ & $f$ & $u$
            \\
            \hline
            $a$ &$a,b$&$a,d$&$a,d$&$a,d$& $a,f$ or $a,f'$& $a,b$
            \\
            \hline
            $c$& $a,b$&--&$d,e$ or $d,e'$& 2 vertices of $W'\setminus \{b\}$& $a,d$&$b,d$
            \\
            \hline
            $d$& $b,d$&--&--&$d,e$ or $d,e'$& $a,d$&$b,d$
            \\
            \hline
        \end{tabular}
  \caption{Case $2(i)$; $b\in W'$.}
  \label{tab:2ibsi}
\end{table}

\begin{table}[!htb]
  \centering 
  \begin{tabular}{r|c|c|c|c|c|c}
            $x\diagdown$\raisebox{1mm}{$y$}& $b$ & $c$ & $d$ & $e$ & $f$ & $u$
            \\
            \hline
            $a$ &$a,b$&$a,d$&$a,d$&$a,d$&$a,d$&$a,b$
            \\
            \hline
            $c$& $a,b$&--&$b,d$& 2 vertices of $W'\setminus \{b\}$& $a,d$&$b,d$
            \\
            \hline
            $d$& $b,d$&--&--& $b,d$&$a,d$&$d,u$ or $d,u'$ or $d'u''$
            \\
            \hline
        \end{tabular}
  \caption{Case $2(ii)$; $b\in W'$.}
  \label{tab:2iibsi}
\end{table}

\medskip 

\noindent
$\bullet$ 
$b\notin W'$. Then $e,f\in W'$ because $W'$ is a dense FTRS in $G'$; see Figures~\ref{fig:casodiag6_2} and~\ref{fig:casodiag6_4}. Let $W=W'\cup \{a,c\}$. Tables~\ref{tab:2ibno} and~\ref{tab:2iibno} show the pairs of vertices that resolve the pair $x,y$, for 
$x\in \{a,c,d\}$, in each case.

\begin{table}[!htb]
  \centering 
  \begin{tabular}{r|c|c|c|c|c|c}
            $x\diagdown$\raisebox{1mm}{$y$}& $b$ & $c$ & $d$ & $e$ & $f$ & $u$
            \\
            \hline
            $a$ &$a,c$&$a,c$&$a,c$&$a,e$&$a,f$&$a,c$
            \\
            \hline
            $c$& $a,c$&--&$c,e$&$c,e$&$c,f$&$c,u$ or $c,u'$ or $c,u''$
            \\
            \hline
            $d$& $a,f$&--&--&$c,e$&$c,f$&$c,u$ or $c,u'$ or $c,u''$
            \\
            \hline
        \end{tabular}
  \caption{Case $2(i)$; $b\not \in W'$.}
  \label{tab:2ibno}
\end{table}

\begin{table}[!htb]
  \centering 
\begin{tabular}{r|c|c|c|c|c|c}
            $x\diagdown$\raisebox{1mm}{$y$}& $b$ & $c$ & $d$ & $e$ & $f$ & $u$
            \\
            \hline
            $a$ &$a,c$&$a,c$&$a,c$&$a,e$&$a,f$&$a,u$ or $a,u'$ or $a,u''$
            \\
            \hline
            $c$& $a,c$&--&$a,c$&$c,e$&$c,f$&$c,u$ or $c,u'$ or $c,u''$
            \\
            \hline
            $d$& $a,f$&--&--&$e,f$&$c,f$&$c,u$ or $c,u'$ or $c,u''$
            \\
            \hline
        \end{tabular}
  \caption{Case $2(ii)$; $b\not \in W'$.}
  \label{tab:2iibno}
\end{table}

Therefore, in the four cases shown in Figure~\ref{fig:casodiag6}, the set $W$ defined in each case is a FTRS of $G$, which is dense by construction and has size   $\cnp+2=\cn$.

\bigskip

\noindent
{\bf Cases $3(i)$, $3(ii)$.}

\medskip

Suppose that the vertices are labeled as shown in Figure~\ref{fig:casodiag7-i-ii} (notice that $f'=g'$ if $G$ has order 8). Denote by $u$ a vertex not in $\{a,b,c,d,e,f,g\}$ and let $u',u''$ be its two neighbors on the boundary of $G'$.
We proceed as in the previous cases by considering the MOP $G'=G- \{a,b,d\}$ of order $n-3$ and taking a dense FTRS $W'$ of $G'$ of size  $\cnp$. We analyze several cases according to whether $c$ and $e$ belong to $W'$ or not.

\begin{figure}[!htb]
\centering
	\subfigure[$3(i)$; $c\not\in W'$.]{
        \includegraphics[scale=0.8, page=48]{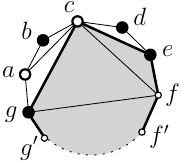}
		\label{fig:casodiag7-i-ii_1}
	}~~~~~
    	\subfigure[$3(i)$; $c \in W'$; $e\not\in W'$.]{
        \includegraphics[scale=0.8, page=49]{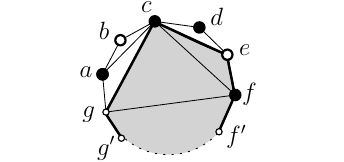}
		\label{fig:casodiag7-i-ii_2}
	}~~~
    	\subfigure[$3(i)$; $c,e\in W'$.]{
        \includegraphics[scale=0.8, page=50]{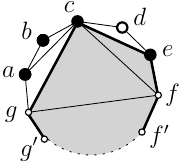}
		\label{fig:casodiag7-i-ii_3}
    }\\
    	\subfigure[$3(ii)$; $c\not\in W'$.]{
        \includegraphics[scale=0.8, page=51]{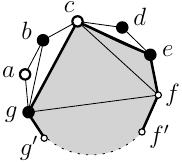}
		\label{fig:casodiag7-i-ii_4}
    }~~~~~~~~
    	\subfigure[$3(ii)$; $c \in W'$; $e\not\in W'$.]{
        \includegraphics[scale=0.8, page=52]{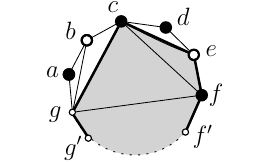}
		\label{fig:casodiag7-i-ii_5}
    }~~~~~~
    	\subfigure[$3(ii)$; $c,e\in W'$.]{
        \includegraphics[scale=0.8, page=53]{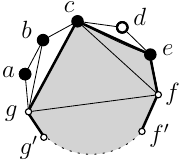}
		\label{fig:casodiag7-i-ii_6}
	}
    \caption{
    Cases $3(i)$ and $3(ii)$. In gray, the MOP $G'=G- \{a,b,d\}$.  Black vertices belong to $W$.}
    \label{fig:casodiag7-i-ii}
\end{figure}

\medskip

\noindent
$\bullet$
$c\notin W'$. Then $e,g\in W'$; see Figures~\ref{fig:casodiag7-i-ii_1} and~\ref{fig:casodiag7-i-ii_4}. Let $W=W'\cup\{b,d\}$. Tables~\ref{tab:3icno} and~\ref{tab:3iicno} show the pairs of vertices that resolve the pair $x,y$, for 
$x\in \{a,b,d\}$, in each of Cases $3(i)$ and $3(ii)$. 

\begin{table}[!htb]
  \centering 
\begin{tabular}{r|c|c|c|c|c|c|c}
            $x\diagdown$\raisebox{1mm}{$y$}& $b$ & $c$ & $d$ & $e$ & $f$ & $g$ & $u$
            \\
            \hline
            $a$ &$b,g$&$d,e$&$d,e$ &$d,e$&$b,e$ &$b,g$ &$b,d$
            \\
            \hline
            $b$& -- &$b,d$ &$b,d$ &$b,e$ &$b,e$ &$b,g$ &$b,d$
            \\
            \hline
            $d$& --&$b,d$&--&$d,e$ &$d,g$ & $d,g$&$b,d$
            \\
            \hline
        \end{tabular}
  \caption{Case $3(i)$; $c\not \in W'$.}
  \label{tab:3icno}
\end{table}

\begin{table}[!htb]
  \centering 
\begin{tabular}{r|c|c|c|c|c|c|c}
            $x\diagdown$\raisebox{1mm}{$y$}& $b$ & $c$ & $d$ & $e$ & $f$ & $g$ & $u$
            \\
            \hline
            $a$ & $b,d$& $d,e$& $b,d$& $b,e$& $b,e$ &$d,g$& $b,u$ or $b,u'$ or $b,u''$
            \\
            \hline
            $b$& -- & $b,d$& $b,d$& $b,e$& $b,e$& $b,g$& $b,d$
            \\
            \hline
            $d$& --&$b,d$&--& $d,e$& $d,g$& $d,g$& $d,e$
            \\
            \hline
        \end{tabular}
  \caption{Case $3(ii)$; $c\not \in W'$.}
  \label{tab:3iicno}
\end{table}
\medskip

\noindent
$\bullet$ 
$c\in W'$ and $e\notin W'$. Then $f\in W'$; see Figures~\ref{fig:casodiag7-i-ii_2} and~\ref{fig:casodiag7-i-ii_5}. Let $W=W'\cup\{a,d\}$. Tables~\ref{tab:3icsi} and~\ref{tab:3iicsi} show the pairs of vertices that resolve the pair $x,y$, for
$x\in \{a,b,d\}$, in each of Cases $3(i)$ and $3(ii)$.

\begin{table}[!htb]
  \centering 
\begin{tabular}{r|c|c|c|c|c|c|c}
            $x\diagdown$\raisebox{1mm}{$y$}& $b$ & $c$ & $d$ & $e$ & $f$ & $g$ & $u$
            \\
            \hline
            $a$ & $a,g$ or $a,g'$& $a,c$& $a,d$& $a,d$&$a,f$ & $a,f$& $a,c$
            \\
            \hline
            $b$& -- & $c,d$& $a,d$& $d,f$& $a,f$& $f,g $ or $f,g'$& $a,c$
            \\
            \hline
            $d$& --&$c,d$&--& $d,f$& $d,f$& $a,d$&$c,d$
            \\
            \hline
        \end{tabular}
  \caption{Case $3(i)$; $c\in W'$; $e\not \in W'$.}
  \label{tab:3icsi}
\end{table}

\begin{table}[!htb]
  \centering 
\begin{tabular}{r|c|c|c|c|c|c|c}
            $x\diagdown$\raisebox{1mm}{$y$}& $b$ & $c$ & $d$ & $e$ & $f$ & $g$ & $u$
            \\
            \hline
            $a$ & $a,c$& $a,c$& $a,d$& $a,d$& $a,f$& $a,c$& $a,u$ or $a,u'$ or $a,u''$
            \\
            \hline
            $b$& -- & $a,c$& $a,d$& $a,d$& $a,f$& $f,g$ or $f,g'$& $a,c$
            \\
            \hline
            $d$& --&$c,d$&--& $d,f$& $d,f$& $a,d$& $c,d$
            \\
            \hline
        \end{tabular}
  \caption{Case $3(ii)$; $c\in W'$; $e\not \in W'$.}
  \label{tab:3iicsi}
\end{table}
\medskip

\noindent
$\bullet$  
$c,e\in W'$. Let $W=W'\cup \{a,b\}$; see Figures~\ref{fig:casodiag7-i-ii_3} and~\ref{fig:casodiag7-i-ii_6}. Tables~\ref{tab:3icesi} and~\ref{tab:3iicesi} show the pairs of vertices that resolve the pair $x,y$, for $x\in \{a,b,d\}$, 
in each of the two cases.

\begin{table}[H]
  \centering 
\begin{tabular}{r|c|c|c|c|c|c|c}
            $x\diagdown$\raisebox{1mm}{$y$}& $b$ & $c$ & $d$ & $e$ & $f$ & $g$ & $u$
            \\
            \hline
            $a$ & $a,b$& $a,c$& $a,b$& $a,e$& $a,e$& $a,b$&$a,b$
            \\
            \hline
            $b$& -- & $b,c$& $a,b$& $b,e$& $a,b$& $b,g$ or $b,g'$& $a,b$
            \\
            \hline
            $d$& --&$b,c$&--& $e,f$ or $e,f'$& 2 vertices of $W'\setminus\{c,e\}$& $a,e$& $c,e$
            \\
            \hline
        \end{tabular}
  \caption{Case $3(i)$; $c, e\in W'$.}
  \label{tab:3icesi}
\end{table}

\begin{table}[H]
  \centering 
\begin{tabular}{r|c|c|c|c|c|c|c}
            $x\diagdown$\raisebox{1mm}{$y$}& $b$ & $c$ & $d$ & $e$ & $f$ & $g$ & $u$
            \\
            \hline
            $a$ & $a,b$& $a,c$& $a,c$& $a,e$& $a,b$& $a,c$& $a,b$
            \\
            \hline
            $b$& -- & $b,c$& $a,b$& $b,e$& $b,e$& $b,g$ or $b,g'$& $a,b$
            \\
            \hline
            $d$& --&$b,c$&--& $e,f$ or $e,f'$& 2 vertices of $W'\setminus\{c,e\}$& $a,e$& $c,e$
            \\
            \hline
        \end{tabular}
  \caption{Case $3(ii)$; $c, e\in W'$.}
  \label{tab:3iicesi}
\end{table}
\medskip

In the six cases shown in Figure~\ref{fig:casodiag7-i-ii}, the set $W$ defined in each case is a FTRS of $G$. In addition, $W$ is dense by construction and has size   $\cnp+2=\cn$.
\bigskip

\noindent
{\bf Case  $3(iii)$.}

\medskip

Suppose that the vertices are labeled as shown in Figure~\ref{fig:casodiag7-iii} (notice that $f'=g'$ if $G$ has order 8). Denote by $u$ a vertex not in $\{a,b,c,d,e,f,g\}$ and let $u'$ and $u''$ be its neighbors on the boundary of $G'$. Let $W'$ be a dense FTRS of size   $\cnp$ of the MOP $G'=G- \{a,b,e\}$ of size $n-3$.
We give a FTRS $W$ of $G$, depending on whether $c$ and $d$ belong to $ W'$ or not.

\begin{figure}[!htb]
\centering
	\subfigure[$3(iii)$; $c\not\in W'$.]{
        \includegraphics[scale=0.8, page=54]{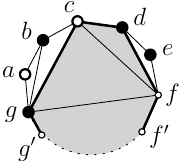}
		\label{fig:casodiag7-iii_1}
	}~~~~~~~~~
    	\subfigure[$3(iii)$; $c \in W'$; $d\not\in W'$.]{
        \includegraphics[scale=0.8, page=55]{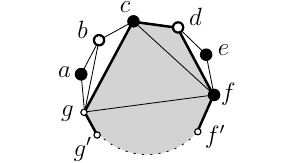}
		\label{fig:casodiag7-iii_2}
	}~~~~
    	\subfigure[$3(iii)$; $c,d\in W'$.]{
        \includegraphics[scale=0.8, page=56]{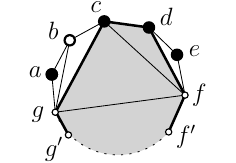}
		\label{fig:casodiag7-iii_3}
	}
    \caption{
    Case $3(iii)$. In gray, the MOP $G'=G- \{a,b,e\}$. Black vertices belong to $W$.}
    \label{fig:casodiag7-iii}
\end{figure}

\medskip

\noindent
$\bullet$ 
$c\notin W'$. Then $d,g\in W'$; see Figure~\ref{fig:casodiag7-iii_1}. Let $W=W'\cup\{b,e\}$. Table~\ref{tab:3iiicno} gives the pairs of vertices that resolve the pair $x, y$, for 
$x\in \{a,b,e\}$. 

\begin{table}[!htb]
  \centering 
    \begin{tabular}{r|c|c|c|c|c|c|c}
            $x\diagdown$\raisebox{1mm}{$y$}& $b$ & $c$ & $d$ & $e$ & $f$ & $g$ & $u$
            \\
            \hline
            $a$ & $b,d$& $d,e$& $d,e$& $d,e$& $d,e$& $e,g$&  $b,u$ or $b,u'$ or $b,u''$
            \\
            \hline
            $b$& -- & $b,d$& $b,d$& $b,e$& $b,e$& $b,g$& $b,u$ or $b,u'$ or $b,u''$
            \\
            \hline
            $e$& --&$b,e$&$d,e$&--& $e,g$& $e,g$& $d,e$
            \\
            \hline
        \end{tabular}
  \caption{Case $3(iii)$; $c \not\in W'$.}
  \label{tab:3iiicno}
\end{table}
\medskip

\noindent
$\bullet$ 
$c\in W'$ and $d\notin W'$. Then, $f\in W'$; see Figure~\ref{fig:casodiag7-iii_2}. Let $W=W'\cup\{a,e\}$. Table~\ref{tab:3iiicsi} gives the pairs of vertices that resolve the pair  $x,y$, for 
$x\in \{a,b,e\}$.

\begin{table}[!htb]
  \centering 
\begin{tabular}{r|c|c|c|c|c|c|c}
            $x\diagdown$\raisebox{1mm}{$y$}& $b$ & $c$ & $d$ & $e$ & $f$ & $g$ & $u$
            \\
            \hline
            $a$ & $a,c$& $a,c$& $a,c$& $a,e$& $a,f$& $a,c$& $a,u$ or $a,u'$ or $a,u''$
            \\
            \hline
            $b$& -- & $a,c$& $a,e$ & $a,e$& $a,f$& $e,f$& $a,c$
            \\
            \hline
            $e$&--&$c,e$&$c,e$&--& $e,f$& $a,e$& $e,u$ or $e,u'$ or $e,u''$
            \\
            \hline
        \end{tabular}
  \caption{Case $3(iii)$; $c \in W'$; $d \not\in W'$.}
  \label{tab:3iiicsi}
\end{table}

\medskip

\noindent
$\bullet$ 
$c,d\in W'$. Let $W=W'\cup\{a,e\}$; see Figure~\ref{fig:casodiag7-iii_3}.  Table~\ref{tab:3iiicdsi} gives the pairs of vertices that resolve the pair $x, y$, 
$x\in \{a,b,e\}$. 

\begin{table}[!htb]
  \centering 
\begin{tabular}{r|c|c|c|c|c|c|c}
            $x\diagdown$\raisebox{1mm}{$y$}& $b$ & $c$ & $d$ & $e$ & $f$ & $g$ & $u$
            \\
            \hline
            $a$ & $a,c$& $a,c$& $a,d$& $a,e$& $a,e$& $a,c$& $a,u$ or $a,u'$ or $a,u''$
            \\
            \hline
            $b$& -- & $a,c$& $a,d$& $a,e$& $a,e$& $e,g$ or $e,g'$& $a,c$
            \\
            \hline
            $e$& --&$c,e$&$d,e$&--& $c,e$& $d,e$& $d,e$
            \\
            \hline
        \end{tabular}
  \caption{Case $3(iii)$; $c,d \in W'$.}
  \label{tab:3iiicdsi}
\end{table}

In the three cases shown in Figure~\ref{fig:casodiag7-iii}, the set $W$ defined in each case is a dense FTRS of $G$ and $|W|=\cnp+2=\cn$. 

\bigskip

In summary, we have built a dense FTRS $W$ of $G$ of size $\cn$ in all cases. As a consequence, $\beta'(G)\le \cn$ for every MOP $G$ of order $n\ge 5$. 
To complete the proof of the theorem, we show that $W$ can be computed in linear time. Based on the previous reasoning, we need to find a diagonal $d_1$ of $G$ defining a MOP of order 5, 6, or 7. Then, after removing three suitable vertices of $G$, say $a,b$ and $c$,  we need to find another diagonal $d_2$ that defines a MOP of order 5, 6, or 7 in $G-\{a,b,c\}$, and so on. To find this set of diagonals $D=\{d_1, d_2, \dots\}$, we use the following observation: A diagonal of a MOP $H$ corresponds to an edge in its dual graph $T_H$, which is a tree whose vertices have degree at most three. Hence, finding a diagonal defining a MOP of order 5, 6, or 7 in $H$ is equivalent to finding an edge in $T_H$ whose deletion defines a component in $T_H$ of order 3, 4, or 5, respectively. 

Thus, the algorithm works as follows. From $G$, the algorithm first builds the dual tree  $T$ of $G$. 
This can be done in linear time. Now, the algorithm explores $T$ in postorder, adding the explored vertices to a list $L$. During the exploration, observe that when the algorithm finds an edge $e$ of $T$ defining a component $C$ of order 3, 4, or 5, the vertices of $C$ are precisely the last vertices added to $L$. Thus, when exploring a new vertex, the algorithm only needs to check whether the last vertices in $L$ define the desired component. If so, the algorithm removes three vertices from $L$ (and from $T$ and $G$) as previously explained, adds the corresponding diagonal of $G$ to $D$, and continues the postorder exploration to search for the next diagonal provided that the updated tree $T$ has at least six vertices (eight vertices in the corresponding updated MOP).
Since these last operations can be performed in constant time and the postorder traversal of $T$ takes linear time, the set of diagonals $D$ can be found in linear time. Finally, exploring the diagonals in $D$ in reverse order, the algorithm can build $W$ by adding two vertices per diagonal, as described in the proof of the theorem. This last exploration requires only linear time, so a FTRS $W$ of cardinality $\cn$ can be found in linear time.
\end{proof}

\section{Conclusions}\label{sec:conclusion}

In this paper, we have studied the FTMD problem for MOPs. We have showed that there are exactly two MOPs $G$ with $\beta'(G)=3$, that $4\le \beta'(G)\le 6$ for any MOP $G$ with $\beta(G)=2$, and that there exists a FTRS of size   $\left\lceil \frac{2n}{3}\right\rceil$ for any MOP $G$ of order $n$.
In this context, several open problems remain. We outline a few of them. 

\begin{problem}
Given a MOP $G$, prove or disprove whether computing the FTMD of $G$ is NP-hard.
\end{problem}

\begin{problem}
Characterize all MOPs $G$ such that $\beta (G)=2$ and $\beta'(G)=k$, for each fixed value of $k\in \{4,5,6\}$.
\end{problem}

\begin{problem}
Characterize all MOPs $G$ such that $\beta (G)=3$ and $\beta'(G)=4$.
\end{problem}

To conclude the paper, we conjecture the following upper bound 
based on the known value of the FTMD of fans. 
A {\em fan} of order $n$, denoted by $F_n$, $n\ge 5$,
is a MOP such that exactly one of the vertices is connected to all the other vertices. 
It is known that $\beta'(F_n)=\left\lceil \frac n2 \right\rceil$, if $n\ge 7$ (see \cite{Cab21,Estr16bis}). 
We suspect that the maximum value of the FTMD of a MOP is attained by these graphs. 

\begin{problem}
Prove or disprove that $\beta'(G) \le \left\lceil \frac{n}{2}\right\rceil$, for every MOP $G$ of order $n\ge 7$.
\end{problem}

\section*{Acknowledgements}

Carmen Hernando, Merc\`e Mora and Javier Tejel acknowledge the support of the project PID2023-150725NB-I00, funded by MICIU/AEI/10.13039/501100011033. Javier Tejel also acknowledges the support of the project E41-23R, funded by Gobierno de Arag\'on. 
Carmen Hernando, Montserrat Maureso and Merc\`e Mora are also supported by grant Gen.Cat.DGR 2021-SGR-00266 from Ag\`encia de Gesti\'o d'Ajuts Universitaris i de Recerca.

\section*{Declaration of competing interest}

The authors declare that they have no known competing financial interests or personal relationships that could have appeared to influence the work reported in this paper.

\bibliographystyle{plainurl}
\bibliography{references}

@article{Ass25,
  title       = {Fault-tolerance in biswapped multiprocessor interconnection networks},
  journal     = {Journal of Parallel and Distributed Computing},
  volume      = {196},
  pages       = {105009},
  year        = {2025},
  doi         = {10.1016/j.jpdc.2024.105009},
  author      = {Basem Assiri and Muhammad F. Nadeem and Waqar Ali and Ali Ahmad},
}

@article{Aze21,
  author       = {Muhammad Azeem and Muhammad F. Nadeem},
  title        = {Metric-based resolvability of polycyclic aromatic hydrocarbons},
  journal      = {The European Physical Journal Plus},
  volume       = {136},
  year         = {2021},
  pages        = {395},
  doi          = {10.1140/epjp/s13360-021-01399-8},
}

@article{BY19,
author = {Bailey, Robert and G. Yero, Ismael},
year = {2019},
month = {01},
pages = {341},
title = {Error-correcting codes from $k$-resolving sets},
volume = {39},
journal = {Discussiones Mathematicae Graph Theory},
doi = {10.7151/dmgt.2087}
}

@article{Bas20,
  title        = {Fault-tolerant metric dimension of circulant graphs {$C_n(1,2,3)$}},
  journal      = {Theoretical Computer Science},
  volume       = {817},
  pages        = {66--79},
  year         = {2020},
  doi          = {10.1016/j.tcs.2019.01.011},
  author       = {Mithun Basak and Laxman Saha and Gour K. Das and Kalishankar Tiwary},
}

@article{Bash21,
  author      = {Humera Bashir and Zohaib Zahid and Agha Kashif and Sohail Zafar and Jia-Bao Liu},
  title       = {On 2-metric resolvability in rotationally-symmetric graphs},
  journal     = {Journal of Intelligent \& Fuzzy Systems},
  volume      = {40},
  number      = {6},
  pages       = {11887--11895},
  year        = {2021},
  doi         = {10.3233/JIFS-210040},
}

@article{Bash22,
  author      = {Bashir, Humera and Zahid, Zohaib and Ojiema, Michael Onyango},
  title       = {Computation of the Fault-Tolerant Metric Dimension of Certain Networks},
  journal     = {Mathematical Problems in Engineering},
  volume      = {2022},
  number      = {1},
  pages       = {6385673},
  doi         = {10.1155/2022/6385673},
  year        = {2022}
}

@article{Cab21,
  title        = {On 2-Resolving Sets in the Join and Corona of Graphs},
  author       = {Jean M. Cabaro and Helen Rara},
  volume       = {14},
  number       = {3},
  pages        = {773--782},
  journal      = {European Journal of Pure and Applied Mathematics},
  year         = {2021},
  doi          = {10.29020/nybg.ejpam.v14i3.3977},
}

@article{Clav21,
  author       = {Merc{\`e} Claverol and Alfredo Garc{\'\i}a and Gregorio Hern{\'a}ndez and Carmen Hernando and Monserrat Maureso and Merc{\`e} Mora and Javier Tejel},
  title        = {Metric dimension of maximal outerplanar graphs},
  journal      = {Bulletin of the {M}alaysian {M}athematical {S}ciences {S}ociety},
  volume       = {44},
  pages        = {2603--2630},
  year         = {2021},
  doi          = {10.1007/s40840-020-01068-6},
}

@article{Corr21,
  title       = {A note on {$k$}-metric dimensional graphs},
  journal     = {Discrete Applied Mathematics},
  volume      = {289},
  pages       = {523--533},
  year        = {2021},
  doi         = {10.1016/j.dam.2020.11.011},
  author      = {Samuel G. Corregidor and {\'A}lvaro Mart{\'\i}nez-P{\'e}rez},
}

@article{Estr15,
  title       = {The {$k$}-Metric Dimension of a Graph},
  journal     = {Applied Mathematics {\&} Information Sciences},
  volume      = {9},
  pages       = {2829--2840},
  year        = {2015},
  doi         = {10.12785/amis/090609},
  author      = {Alejandro Estrada-Moreno and Juan A. Rodr{\'\i}guez-Vel{\'a}zquez and Ismael G. Yero},
}

@article{Estr16,
  author      = {Estrada-Moreno, Alejandro and Rodr\'{\i}guez-Vel\'{a}zquez, Juan A. and Yero, Ismael G.},
  title       = {The {$k$}-metric dimension of the lexicographic product of graphs},
  year        = {2016},
  volume      = {339},
  number      = {7},
  doi         = {10.1016/j.disc.2015.12.024},
  journal     = {Discrete Mathematics},
  pages       = {1924--1934},
}

@article{Estr16bis,
  author      = {Estrada-Moreno, Alejandro and Rodr\'{\i}guez-Vel\'{a}zquez, Juan A. and Yero, Ismael G.},
  title       = {The {$k$}-metric dimension of corona product graphs},
  year        = {2016},
  volume      = {39},
  doi         = {10.1007/s40840-015-0282-2},
  journal     = {Bulletin of the {M}alaysian {M}athematical {S}ciences {S}ociety},
  pages       = {135--154},
}

@article{Yero17,
  title       = {Computing the $k$-metric dimension of graphs},
  journal     = {Applied Mathematics and Computation},
  volume      = {300},
  pages       = {60-69},
  year        = {2017},
  issn        = {0096-3003},
  doi         = {https://doi.org/10.1016/j.amc.2016.12.005},
  author      = {Alejandro Estrada-Moreno and Juan A. Rodr{\'\i}guez-Vel{\'a}zquez and Ismael G. Yero},
}

@article{Fah25,
  author      = {M. Faheem and M. Ahmad and Z. Zahid and Njood. S. E. ALmutire and Hamiden Abd El-Wahed Khalifa and Taha Radwan},
  title       = {Investigating the Metric and Fault-Tolerant Dimensions in Para-Line Network Topologies},
  journal     = {AKCE International Journal of Graphs and Combinatorics},
  volume      = {00},
  number      = {0},
  pages       = {1--12},
  year        = {2025},
  doi         = {10.1080/09728600.2025.2487427},
}

@inproceedings{Gao25,
  author      = {Gao, Suogang and Hou, Bo and Liu, Wen},
  editor      = {Pardalos, Panos M. and Du, Ding-Zhu and Thai, M. T.},
  title       = {The metric dimension of a graph and its variants},
  bookTitle   = {Handbook of Combinatorial Optimization},
  year        = {2025},
  publisher   = {Springer, New York},
  pages       = {1--28},
  doi         = {10.1007/978-1-4614-6624-6_103-1},
}

@article{Gee17,
  title       = {2-Metric dimension of Cartesian product of graphs},
  volume      = {112},
  doi         = {10.12732/ijpam.v112i1.2},
  number      = {1},
  journal     = {International Journal of Pure and Applied Mathematics},
  author      = {K.N. Geetha and B. Sooryanarayana},
  year        = {2017},
  pages       = {27--45},
}

@article{Gen25,
  author      = {Jesse Geneson and Shen-Fu Tsai},
  title       = {Fault tolerance for metric dimension and its variants},
  journal     = {arXiv:2502.02731},
  year        = {2025},
  url         = {https://arxiv.org/abs/2502.02731}
}

@article{Har76,
  author      = {Harary, Frank and  Melter, R.A.},
  title       = {On the metric dimension of a graph},
  journal     = {Ars Combinatoria},
  volume      = {2},
  pages       = {191--195},
  year        = {1976},
}

@article{Hay20,
  author      = {Hayat, Sakander and Khan, Asad and Malik, Muhammad Y. H. and Imran, Muhammad and Siddiqui, Muhammad K.},
  journal     = {IEEE Access},
  title       = {Fault-Tolerant Metric Dimension of Interconnection Networks},
  year        = {2020},
  volume      = {8},
  pages       = {145435--145445},
  doi={10.1109/ACCESS.2020.3014883},
}

@inproceedings{Her08,
  title       = {Fault-tolerant metric dimension of graphs},
  author      = {Carmen Hernando and Merce Mora and Slater, Peter J. and David Wood},
  year        = {2008},
  pages       = {81--85},
  editor      = {Manoj Changat and Sandi Klavzar and A Vijayakumar},
  booktitle   = {Convexity in {D}iscrete {S}tructures},
  publisher   = {International Press of Boston, Inc.},
  address     = {USA},
}

@article{Huo20,
  title       = {On the 2-metric resolvability of graphs},
  volume      = {5},
  doi         = {10.3934/math.2020425},
  number      = {6},
  journal     = {AIMS Mathematics},
  author      = {Chenggang Huo and Humera Bashir and Zohaib Zahid and Yu Ming Chu},
  year        = {2020},
  pages       = {6609--6619},
}

@article{Jav09,
  author      = {Javaid, Imran and Salman, Muhammad and Chaudhry, Muhammad A. and  Shokat, Sara },
  title       = {Fault-tolerance in resolvability},
  journal     = {Utilitas {M}athematica},
  volume      = {80},
  pages       = {263--275},
  year        = {2009},
}

@article{Khan25,
  title       = {Fault-tolerance and unique identification of vertices and edges in a graph: The fault-tolerant mixed metric dimension},
  journal     = {Journal of Parallel and Distributed Computing},
  volume      = {197},
  pages       = {105024},
  year        = {2025},
  doi         = {10.1016/j.jpdc.2024.105024},
  author      = {Asad Khan and Sikander Ali and Sakander Hayat and Muhammad Azeem and Yubin Zhong and Manzoor Ahmad Zahid and Mohammed J.F. Alenazi},
}

@article{Knor25,
  author      = {Martin Knor and Jelena Sedlar and Riste \v{S}krekovski},
  title       = {Fault tolerance of metric basis can be expensive},
  journal     = {Mediterranean Journal of Mathematics},
  year        = {2025},
  volume      = {22},
  pages       = {148},
  doi         = {10.1007/s00009-025-02894-3},
}

@article{Liu21,
  volume      = {6},
  doi         = {10.3934/math.2021069},
  number      = {2},
  journal     = {AIMS Mathematics},
  year        = {2021},
  pages       = {1140--1152},
  author      = {Xiaogang Liu and Muhammad Ahsan and Zohaib Zahid and Shuili Ren},
  title       = {Fault-tolerant edge metric dimension of certain families of graphs},
}

@article{Man25,
  author      = {Manikonda, Gayathri and Bhat, Vijay Kumar and Sharma, Sunny Kumar},
  title       = {Unique fault-tolerance resolvability codes for certain generalized convex polytopes with applications},
  journal     = {Discrete Mathematics, Algorithms and Applications},
  volume      = {0},
  number      = {0},
  pages       = {2550013},
  year        = {2025},
  doi         = {10.1142/S1793830925500132},
}

@article{MST22,
  title       = {Resolving sets tolerant to failures in three-dimensional grids},
  journal     = {Mediterranean Journal of Mathematics},
  volume      = {19},
  number      = {188},
  pages       = {},
  year        = {2022},
  doi         = {10.1007/s00009-022-02096-1},
  author      = {Merc\`e Mora and Mar\'ia Jos\'e Souto-Salorio and Ana D. Tarr\'io-Tobar},
}

@article{Nad22,
  title       = {On Metric Dimension and Fault-Tolerant Metric Dimension of Some Chemical Structures},
  journal     = {Polycyclic Aromatic Compounds},
  volume      = {42},
  number      = {10},
  pages       = {6975--6987},
  year        = {2022},
  doi         = {10.1080/10406638.2021.1994429},
  author      = {Muhammad F. Nadeem and Ayesha Shabbir and Muhammad Azeem},
}

@article{Pra22,
  title       = {Twin vertices in fault-tolerant metric sets and fault-tolerant metric dimension of multistage interconnection networks},
  journal     = {Applied Mathematics and Computation},
  volume      = {420},
  pages       = {126897},
  year        = {2022},
  doi         = {10.1016/j.amc.2021.126897},
  author      = {Savari Prabhu and V. Manimozhi and M. Arulperumjothi and Sandi Klav{\v{z}}ar},
}

@article{Rah22,
  title       = {The {$k$}-metric Dimension of Double Fan Graph},
  volume      = {2},
  doi         = {10.14421/quadratic.2022.021-05},
  number      = {1},
  journal     = {Quadratic: Journal of Innovation and Technology in Mathematics and Mathematics Education},
  author      = {Deddy Rahmadi and Yeni Susanti},
  year        = {2022},
  pages       = {46--51}
}

@article{Saha21,
  doi         = {10.1088/1742-6596/1714/1/012029},
  year        = {2021},
  volume      = {1714},
  number      = {1},
  pages       = {012029},
  author      = {Saha, Laxman},
  title       = {Fault-Tolerant Metric Dimension of Cube of Paths},
  journal     = {Journal of Physics: Conference Series},
}

@Article{Saha22,
  author      = {Saha, Laxman and Lama, Rupen and Tiwary, Kalishankar and Das, Kinkar Chandra and Shang, Yilun},
  title       = {Fault-Tolerant Metric Dimension of Circulant Graphs},
  journal     = {Mathematics},
  volume      = {10},
  year        = {2022},
  number      = {1},
  pages       = {124},
  doi         = {10.3390/math10010124}
}

@article{Sar24,
  author      = {Muhammad S. Sardar and Komal Rasheed and Murat Cancan and Mohammad R. Farahani and Mehdi Alaeiyan and Shobha V. Patil},
  title       = {Fault Tolerant Metric Dimension of Arithmetic Graphs},
  journal     = {Journal of Combinatorial Mathematics and Combinatorial Computing},
  year        = {2024},
  volume      = {122},
  pages       = {13--32},
  doi         = {10.61091/jcmcc122-02},
}

@article{Sch23,
  title       = {A note on the complexity of {$k$}-metric dimension},
  journal     = {Applied Mathematics and Computation},
  volume      = {457},
  pages       = {128204},
  year        = {2023},
  doi         = {10.1016/j.amc.2023.128204},
  author      = {Yannick Schmitz and Duygu Vietz and Egon Wanke},
}

@article{Shar25,
  title       = {Fault-tolerant resolvability of double antiprism and its related graphs},
  journal     = {Annals of the University of Craiova},
  volume      = {52},
  number      = {1},
  pages       = {170--186},
  year        = {2025},
  doi         = {0.52846/ami.v52i1.1939},
  author      = {Sunny Kumar Sharma and Vijay Kumar Bhat},
}

@article{Sher91,
  title       = {Covering and guarding polygons using {$L_k$}-sets},
  journal     = {Geometriae Dedicata},
  volume      = {37},
  pages       = {183--203},
  year        = {1991},
  doi         = {10.1007/BF00147413},
  author      = {Thomas Shermer},
}

@article{Sid19,
  title       = {Resolvability and fault-tolerant resolvability structures of convex polytopes},
  journal     = {Theoretical Computer Science},
  volume      = {796},
  pages       = {114-128},
  year        = {2019},
  doi         = {10.1016/j.tcs.2019.08.032},
  author      = {Hafiz M. A. Siddiqui and Sakander Hayat and Asad Khan and Muhammad Imran and Ayesha Razzaq and Jia-Bao Liu},
}

@article{Slater,
  author      = {Slater, Paul J.},
  title       = {Leaves of trees},
  journal     = {Congressus Numerantium},
  volume      = {14},
  pages       = {549--559},
  year        = {1975},
}

@article{Som19,
  title       = {Fault-tolerant resolvability of oxide interconnections},
  journal     = {International Journal of Innovative Technology and Exploring Engineering},
  volume      = {8},
  number      = {12},
  pages       = {1778--1784},
  year        = {2019},
  doi         = {10.35940/ijitee.L3245.1081219},
  author      = {Somasundari, M. and Raj, F. Simon},
}

@article{Till23,
  author      = {Tillquist, Richard C. and Frongillo, Rafael M. and Lladser, Manuel E.},
  title       = {Getting the Lay of the Land in Discrete Space: A Survey of Metric Dimension and Its Applications},
  journal     = {SIAM Review},
  volume      = {65},
  number      = {4},
  pages       = {919--962},
  year        = {2023},
  doi         = {10.1137/21M1409512},
}

@article{Wang21,
  author      = {Wang, Hua and Azeem, Muhammad and Nadeem, Muhammad F. and Ur-Rehman, Ata and Aslam, Adnan},
  title       = {On Fault-Tolerant Resolving Sets of Some Families of Ladder Networks},
  journal     = {Complexity},
  volume      = {2021},
  number      = {1},
  pages       = {9939559},
  doi         = {10.1155/2021/9939559},
  year        = {2021},
}

@article{Zhe21,
  author      = {Zheng, Zhi-Bo and Ahmad, Ashfaq and Hussain, Zaffar and Munir, Mobeen and Qureshi, Muhammad I. and Ali, Imtiaz and Liu, Jia-Bao},
  title       = {Fault-Tolerant Metric Dimension of Generalized Wheels and Convex Polytopes},
  journal     = {Mathematical Problems in Engineering},
  volume      = {2020},
  number      = {1},
  pages       = {1216542},
  doi         = {10.1155/2020/1216542},
  year        = {2020},
}

\end{document}